\documentclass[11pt, reqno]{amsart}
\usepackage{amsthm}   
\usepackage{amsfonts, amsmath, amscd}
\usepackage{amssymb, latexsym}

\usepackage[all,cmtip]{xy}     %%%  to adjust spacing between rows in xypic

\usepackage{enumerate} 

\usepackage{color}
\usepackage[usenames,dvipsnames]{xcolor}
\usepackage{verbatim}
 
\usepackage{hyperref}
 \hypersetup{colorlinks=false}
 \usepackage{cite}

 \usepackage{graphicx, psfrag}
 \usepackage{pstool}

    \usepackage{xspace}   %%%% make spacing correct after invoking a newcommand
    \usepackage{tikz-cd}
\usepackage{tikz}
\usetikzlibrary{arrows}

\usepackage[titletoc]{appendix}
\usepackage[normalem]{ulem}     %%%  for editing only
\usepackage{mathrsfs}    % \usepackage[mathscr]{euscript}      %%%%  for nicer script fonts

\def\bear{\begin{eqnarray}\begin{aligned}}
\def\eear{\end{aligned}\end{eqnarray}}
\def\best{\begin{eqnarray*}}
\def\eest{\end{eqnarray*}}

\renewcommand{\theequation}{\arabic{section}.\arabic{equation}}
\newtheorem{theorem}{Theorem}[section]
\newtheorem{prop}[theorem]{Proposition}

\newtheorem{lemma}[theorem]{Lemma}
\newtheorem{cor}[theorem]{Corollary}
\newtheorem{defn}[theorem]{Definition}

\newtheorem{remark}[theorem]{Remark}
\newenvironment{rem}{\begin{remark}\rm}{\end{remark}}

\newtheorem{example}[theorem]{Example}
\newenvironment{ex}{\begin{example}\rm}{\end{example}}

\let \oldsection \section
\renewcommand{\section}{\vspace{3mm}\oldsection}

\renewcommand{\theequation}{\thesection.\arabic{equation}}

\def\ra{\rightarrow}

\def\xra{\xrightarrow} 
\def\rg{\rangle}
\def\lg{\langle}
\def\ti{\times}

\def\bd{\partial}
  \def\diam{{\rm diam}}
  \def\dist{\mbox{dist}}
\def\dim{\mathrm{dim}\, }
\def\ind{\mathrm{index}}
\def\Hom{\mathrm{Hom}}

\def\ma#1{\mathop {#1} \limits}
\newcommand{\Setminus}{\!\setminus\! }

\def\al{\alpha}
\def\ep{\varepsilon}
\def\phi{\varphi}
\def\si{\sigma}

\def\Bbb{\mathbb}

\def\R{{ \Bbb R}}
\def\Q{{ \Bbb Q}}

\def\Z{{ \Bbb Z}}

\def\cal{\mathcal}
\def\A{{\cal A}}

\def\F{\cal F}
\def\M{{\cal M}}
\def\N{{\cal N}}

\def\P{{ \cal  P}}

\def\SS{\cal S}

\def\V{{\mathscr V}}
\def\oV{\overline{\mathscr V}}
\def\X{\cal X}

\def\wh#1{\widehat{#1}}
\def\ov#1{\overline{#1}}

 \def\oM{\overline{M}}
 
\def\Sa{{\mathcal S}_\alpha}

\def\Sap{{\cal S}_{\alpha, p}}

\def\KP{{\mathcal K}_{\mathcal P}}
\def\Ab{{\mathcal{A}b}}

 \newcommand\cHH{\check{\mathrm{H}}}  %  for cech homology
 \newcommand \sHH{{}^s{\mathrm{H}}}  %  for Steenrod homology
\newcommand\Cech{\v{C}ech\xspace }

\def\AC{{\cal A}_{\mbox{\tiny $C$}}}
\def\ACM{{\cal A}_{\mbox{\tiny $CM$}}}
\def\ALC{{\cal A}_{\mbox{\tiny $LC$}}}
\def\AEC{{\cal A}_{\mbox{\tiny $EC$}}}
\def\HBM{H^{\mbox{\tiny $BM$}}}
\def\rfc{{r\!f\!c}}

\title{\bf  Relating VFCs  on thin compactifications\vskip.2in}
\author{Eleny-Nicoleta Ionel}
\address{Stanford University, Palo Alto, CA, USA.}
\email{ionel@math.stanford.edu}

\author{Thomas H. Parker}
\address{Michigan State University, East Lansing MI, USA. }
\email{parker@math.msu.edu}

\begin{document}

\begin{abstract}
Many moduli spaces that occur in geometric analysis admit ``Fredholm-stratified thin compactifications''  in the sense of \cite{IPThin} and hence admit a relative fundamental class (RFC), also as defined in \cite{IPThin}.  We extend these results, emphasizing the naturality  of the RFC, eliminating the need for a stratification, and proving  three compatibility results: the invariants defined by the RFC agree with those defined by pseudo-cycles, the RFC is compatible with cutdown moduli spaces, and the RFC agrees with the virtual fundamental class  (VFC) constructed by J. Pardon via implicit atlases  in  all cases where both are defined. 
\end{abstract}

\maketitle

 In  symplectic Gromov-Witten  theory and in gauge theories, the central object  of study is a ``universal moduli space''.    It arises by considering  the moduli space $\M$ of all pairs $(\phi, p)$ where $\phi$ is a solution, modulo gauge, of some family of non-linear elliptic PDEs  parameterized  by the elements $p$ in a Banach space $\P$.  This space $\M$ embeds in one or more (relative) ``compactifications'' $\ov\M$, creating a diagram  
\bear
\label{0.1}
\begin{tikzcd}  
\M\ar[dr, "\pi"] \arrow[hook]{r} & \ov\M \ar[d, "\ov\pi"] \\
& \P, 
\end{tikzcd}  
\eear
where $\pi$ is differentiable and $\ov\pi$ is proper and continuous.    This   determines an entire category whose objects are pullback diagrams (``proper base changes'')
\begin{equation}\label{0.2}
\xymatrix@=5mm{
\ov{\M}_\sigma \ar[d] \ar[r] & \ov{\M} \ar[d]^\pi  \\
K\ar[r]_\sigma & P,}
\end{equation}
where $\sigma$ is a proper continuous map from a path-connected space.  Following   \cite{IPThin},  a {\em relative fundamental class} is a functor on this category that  associates to each diagram  \eqref{0.2}  a  \v{C}ech homology class
 $$
[\ov\M_\sigma]^{\rfc}\in \cHH_*(\ov\M_\sigma)
$$
that satisfies a normalization condition.   The normalization consists of the requirement that the fiber $\ov\M_p=\pi^{-1}(p)$ over  a dense set of $p\in \P$ is a ``thinly compactified manifold''  and that, taking $\sigma$ to be the inclusion of $p$,  $[\ov\M_{p}]^{\rfc}$ is the  fundamental class $[\ov\M_{p}]$.  As in Sections 1 and 2, one can use the continuity property of \Cech homology to show that there is a unique induced  class in the \Cech  homology of the fiber $\ov\M_q$ over every point $q$ in $\P$, which is the  conventional viewpoint on what a virtual fundamental class should be.
Naturality implies that the relative fundamental  class is invariant under deformations.

The existence of an RFC  was proved in \cite{IPThin} under the assumption that $\ov\M$ is a ``Fredholm stratified thin compactification''.   The first aim of this paper is to strengthen the existence result by replacing the stratification requirement by a covering condition which is much easier to show in applications (see Definition~\ref{1.Def1.1}).

 There are  several existing and proposed methods of defining invariants  in  geometric analysis approaches to Gromov-Witten theory and other gauge theories. Some of these involve
constructing a ``virtual fundamental class'' (VFC)  similar to the  RFC above.  Here, we  consider three:
\begin{enumerate} \setlength{\itemsep}{4pt}
\item The  standard definition  using pseudo-cycles \cite{rt1}, \cite{ms2}.
\item The signed count of elements in a  0-dimensional cutdown moduli space.
%in $\ov\M$ obtained imposing contraints to define a cutdown moduli space inside  $\ov\M$.
\item The VFC defined by J. Pardon using implicit atlases \cite{pardon}.
\end{enumerate}
These vary  in the techniques needed for their construction, and in the required assumptions on the space $\ov\M$.    Only a few results are known  about how they are related.   The second goal of this paper is to establish an initial ``common setting'' in which these definitions are comparable and equal.

\medskip

This paper has four parts and two appendices.

\smallskip

{\bf \sc Part I.}   Sections~1 and 2 review and extend the results of \cite{IPThin} on the existence of a relative fundamental class.  As in \cite{IPThin}, the construction involves a combination of Steenrod and \Cech homologies, together with the Sard-Smale theorem.   The presentation  emphasizes the naturality properties,   and is  nearly self-contained, requiring only the Extension  Lemma~3.4  from \cite{IPThin}.   The main result is Theorem~\ref{theorem1.2}:  if a compactified moduli space \eqref{0.1} is a  ``Fredholm thin compactification''  as in Definition~\ref{1.Def1.1}, then it admits a unique relative fundamental class.   As a result, Definition~\ref{1.Def1.1}  delineates the structure that must be verified in applications.

This approach establishes the existence of an RFC with  no need for gluing theorems at the boundary of $\M$, and no need to repeatedly make cobordism arguments, which now follow from the naturality of the RFC.

\smallskip

{\bf \sc Part 2.}    Sections~3-5 relate   the numerical invariants \eqref{1.invts} defined by an RFC  to  intersections of  pseudo-cycles.  This is useful because many gauge theory invariants were originally defined using intersections of pseudo-cycles, as was done for  the  Gromov-Witten invariants of symplectic manifolds  in \cite{rt1} and \cite{ms2}.  

Section~3 shows that a pseudo-cycle determines a Steenrod homology class and relates it, under appropriate assumptions, to the pushforward of the fundamental class of a thin compactification (Lemmas~\ref{pseudocycleclass} and \ref{pseudocycleTheorem}).   Section 4 shows that the geometric and homological intersection numbers of pseudo-cycles are equal (Proposition~\ref{Dot=dot}).    Sections 5   moves to the family setting and establishes our first compatibility result.  Theorem~\ref{theorem.A} states that,   for  maps from a  Fredholm thin compactification, the 
pushforward of the RFC is equal in \Cech homology  to the  pseudo-cycle class defined by the images of generic fibers.   In particular, the
numerical invariants \eqref{1.invts} defined by the RFC  are equal to those defined by  the pseudo-cycle  approach.

\smallskip

{\bf \sc Part 3.}  In Gromov-Witten and gauge theories, one often uses maps (or sections of bundles) to impose constraints;  these determine subsets of the moduli space often called ``cutdown'' moduli spaces, defined as  the inverse image of a submanifold $V$  cf. \eqref{5.cutdown}.  We show that if $f$ is  (fully)  transverse to $V$, the cutdown family is itself a Fredholm thin compactification, thus carries an RFC.

  Algebraic geometers have a formula, involving the cap product with a Thom class, that relates the virtual fundamental classes of the original moduli space and the cutdown space.
Section~6 establishes  our second compatibility result:   Proposition~\ref{P.6.cut.down} shows that the  exact same formula applies for our RFC.

\smallskip

{\bf \sc Part 4.}  Sections~7 and 8 relate our RFC with the virtual fundamental class (VFC) constructed by J.~Pardon.   Pardon's construction applies to spaces that admit an ``implicit atlas'' as is defined in \cite{pardon}.   In particular, an implicit atlas contains a special chart, called the regular locus $X^{reg}_\emptyset$, which is an oriented manifold.  Section~7 proves the fact, communicated to us by Pardon, that the restriction of the VFC to $X^{reg}_\emptyset$ is the fundamental class of $X^{reg}_\emptyset$.
 This result -- Lemma~\ref{Pardon.L} -- is independent of the previous sections, and can be  viewed  as an addendum to Pardon's paper;  the proof requires  familiarity with \cite{pardon}.   It is used in Section~8 to prove our third compatibility result:  our RFC agrees with Pardon's VFC for thin families that admit an implicit atlas (Theorem~\ref{5.maintheorem}).  
\smallskip

Appendix~A  summarizes needed facts about  Steenrod, \Cech and Borel-Moore homologies, including cap products, fundamental classes and intersection pairings.  In fact, we mostly  work in a context where Steenrod and Borel-Moore homologies are isomorphic, and make use of the well-developed literature on intersection theory for Borel-Moore theory.  Appendix~B lists results from Dimension theory that are needed in Sections~1-5.

\medskip

Throughout, all maps are assumed to be continuous, and all spaces are  Hausdorff spaces.  The  term manifold means a Hausdorff topological manifold, with additional properties (e.g. metrizability)  only where explicitly stated at beginning of each section. 
  For simplicity, after Section~1, we use only homology theories with constant coefficients in $R= \Z$ or $\Q$.  One can work consistently with a single homology theory by taking  $R=\Q$, as we note at the end of Section~\ref{section1}.

%%%%%%%%%%%%%%%%%%%%%%%%%%%%%%%%%%%%%%%%%%%%%%%%%%%%%%%%%%%%
%%%%%%%%%%%%%%%%%  Section 1  %%%%%%%%%%%%%%%%%%%%%%%%%%%%%%
%%%%%%%%%%%%%%%%%%%%%%%%%%%%%%%%%%%%%%%%%%%%%%%%%%%%%%%%%%%%%%
%\part{Freholm thin compactified families}

\setcounter{equation}{0}
 \section{Thin compactifications and  families} 
\label{section1}
\medskip

  The notion of a relatively thin family was defined in \cite{IPThin}, and used to define relative fundamental classes.   This section and the next summarize and  generalize the relevant definitions, first for a single space $M$, then for families.    We refer the reader to Sections~2 and 3 of  \cite{IPThin} for further details, and to   
   Appendix~\ref{sectionA},  \cite{ma}, and \cite{ES} for background facts about homology theories.  
   
\smallskip

 Steenrod homology $\sHH_*$ is a homology theory   based on infinite chains; it is defined on the category $\ALC$ of locally compact Hausdorff spaces and proper continuous maps as in \cite[Chapter 4]{ma}.    It satisfies the Eilenberg-Steenrod axioms, including the exactness axiom.  It has three additional properties that distinguish it from singular homology and make it especially well-suited for our purposes:  \\[-1mm]

 $\bullet$\  For  each  open set $U \subseteq X$ there is a  natural ``restriction'' homomorphism 
 \bear\label{1.rhoU}
 \rho_U: \sHH_*(X)\to \sHH_*(U).
 \eear
\indent  $\bullet$\  For each closed set $\iota:A\hookrightarrow X$, there is a natural  long exact sequence 
\bear
\label{1.LES}
 \cdots \longrightarrow \sHH_p(A) \overset{\iota_*}\longrightarrow \sHH_p(X) \overset{\rho}\longrightarrow \sHH_p(X \Setminus A) \overset{\partial}\longrightarrow \sHH_{p-1}(A) \longrightarrow \cdots
\eear
  
\indent  $\bullet$\  Each  oriented $d$-dimensional manifold $M$, whether compact or not,  has a fundamental class 
 $$
 [M]\in \sHH_d(M; G)
 $$
  for any abelian  coefficient group $G$.   Then  open subsets  of $U \subseteq M$ are manifolds with the induced orientation, and
 \bear\label{5.2MN}
\rho_U[M]=[U].
\eear
   
\subsection{Thin compactifications.}    Our starting point is the following definition, which was introduced in \cite{IPThin}.
  
  \begin{defn}\label{Defn1.1}
A {\em thin compactification} of  oriented $d$-dimensional topological manifold $M$ is a compact   Hausdorff space $\overline{M}$ containing $M$ such that  
 $S=\ov{M}\Setminus M$  is a closed subset of codimension~2 in Steenrod homology:
 \bear\label{1.1}
\sHH_k(S)=0 \ \quad \forall\, k> d-2. 
\eear
\end{defn}

   Under   condition \eqref{1.1},  the exact sequence \eqref{1.LES}  for   $(\ov{M}, S)$   shows that restriction to $M$ gives an isomorphism
\bear\label{YYFrseq}
\begin{tikzcd}
\rho_M: \sHH_d(\ov{M};G)  \ar{r}{\cong}&  \sHH_d(M;G).
\end{tikzcd}
\eear
Hence   the fundamental class  $[M]$ extends uniquely  to a   class of $[\ov{M}]\in \sHH_d(\ov{M};G)$   defined by 
 \bear\label{1.fc}
 \rho_M[\ov{M}] = [M].  
\eear
One  advantage of using Steenrod homology is that this extension exists with no assumptions on the differentiability or triangularizability of $\ov{M}$, and without any need to describe the structure of neighborhoods of $\ov{M}\Setminus M$ in $\ov{M}$. 

\medskip

 Similar considerations apply to cobordisms.  A   {\em thin  compactified cobordism} between $\ov{M}_1$ and $\ov{M}_2$ is a compact Hausdorff pair $(\ov W, S)$ such that 
\begin{enumerate}[(i)]
 \setlength{\itemsep}{4pt}
\item $W=\ov W\Setminus S$ is an oriented cobordism between    the $d$-dimensional  manifolds $M_1$ and $M_2$.
\item$\ov M_i\subset \ov W $ is a thin compactification of $M_i$ for $i=1,2$,  and $\ov M_1\cap \ov M_2=\emptyset$.
\item  $\sHH_{k}(S) = 0$ for all  $k\ge d$. \\[-2mm]
\end{enumerate}
 As in  \cite[Corollary 2.12]{IPThin},   it follows that 
\bear\label{1.3}
 (\iota_1)_*[\ov{M}_1] \ =\ (\iota_2)_*[\ov{M}_2] \qquad\mbox{in}\qquad \sHH_d(\ov{W}),
\eear
 where $\iota_i:\ov{M}_i\to \ov{W}$ are the inclusions.

\medskip

One can now pass from Steenrod to \Cech homology.    There is a natural transformation 
\bear\label{1.gamma}
\sHH_d(X) \to \cHH_d(X)
\eear
 defined on the category $\ALC$  (cf. Appendix~\ref{sectionA}).   While \Cech homology does not satisfy the exactness axiom,  it has a property which does not hold for singular homology that we will make crucial use of:
\begin{itemize} \setlength\itemsep{4pt}
\item  \Cech Continuity Property:   For every inverse system of compact Hausdorff spaces with inverse limit $X$, the  maps $X\to X_\al$ induce a natural  isomorphism
\bear\label{1.Cech.ContinuityProperty}
 \cHH_*(X) \overset{\cong}\longrightarrow \varprojlim  \cHH_*(X_\alpha)
 \eear
\end{itemize}
 (cf. \cite[Thm. X.3.1]{ES}).   Applying \eqref{1.gamma} to  the  classes in \eqref{1.fc} yields
 \Cech theory fundamental  classes, which we will usually write as simply
\bear\label{1.Cechfc}
[M] \in \cHH_d(M; G), \hspace{1cm}  [\ov{M}] \in \cHH_d(\ov{M}; G).
\eear
Equations  \eqref{1.fc} and \eqref{1.3} continue to hold in \Cech theory.  (If desired, one could further push these classes into the dual of compactly supported \Cech cohomology; see \S\ref{sSA.cohom} and  Remark~5.0.2 of \cite{pardon}.)

\medskip

 Appendix~B  shows that  one can replace Steenrod by  Borel-Moore homology in  \eqref{1.1} (Lemma~\ref{LemmaB1} with $X=S$).  It also shows that  condition \eqref{1.1} holds if   the  Lebesgue covering dimension  of $S$ satisfies
  \bear\label{1.0}
\dim S \le d-2,
\eear
 and  this is true if $S$ is covered by  manifolds of dimension $\le d-2$ in  the sense of  Lemma~\ref{LemmaA2}.    This   provides a practical method  to verify \eqref{1.1}.  For example, if 
  $M$ is  a smooth complex quasi-projective variety with positive dimension,  then its closure $\ov{M}$ and its 1-point compactification  $M^+$ are both thin compactifications.

\smallskip

\subsection{Families.}  These ideas extend naturally to families.  Consider a  proper continuous map
\bear
 \label{1.8diagram}
\xymatrix{
\ov{\M} \ar[d]^{\ov\pi} \\
\P
}
\eear
which we
regard  as a family of spaces, namely the fibers $\ov\M_p=\ov\pi^{-1}(p)$ for $p\in\P$. 
Let ${\mathcal K}_\P$ denote the category whose objects are  continuous maps $\phi:K\to\P$ where $K$ is a nonempty, compact,  path-connected Hausdorff space, and whose morphisms are commutative diagrams
\bear
 \label{2.diagram1}
\xymatrix@C=10mm@R=.3mm{
K' \ar[dd]_\psi \ar[dr]^{\phi'} & \\
& \P \\
K \ar[ru]_\phi
}
\eear
of   continuous maps.    Each map  $\phi:K\to \P$ in $\KP$ gives a pullback diagram (\ref{2.Cat}a) of proper continuous maps,
where
\bear\label{1.M_K}
\ov\M_\phi\ =\ \big\{ (k,x)\in K\times \ov\M\, \big|\, \phi(k)=\ov\pi(x)\big\}
\eear
 is the fiber product  of $\phi$ and $\ov\pi$. Similarly,   the morphism  \eqref{2.diagram1} gives the diagram (\ref{2.Cat}b).      

\refstepcounter{equation}\label{2.Cat}% Correctly mark and label equation
\begin{equation}
\tag{\theequation a,b}\label{2.Cat1}
\begin{aligned}
\xymatrix@=8mm{
\ov{\M}_\phi  \ar[d]_{\ov\pi_\phi} \ar[r]_{\wh{\phi}} & \ov{\M} \ar[d]^{\ov\pi}   \\
K\ar[r]_\phi & \P
} \end{aligned}
\hspace{3cm}  
\begin{aligned}
\xymatrix@=5mm{
& &   \ov\M \ar[d]^{\ov\pi}\\
\ov\M_{\phi'}  \ar[r]_{\wh{\psi}}  \ar[d] \ar[rru]^{\wh\phi'} &  \ov\M_\phi  \ar[d] \ar[ru]_{\!\!\wh\phi}&  \P   \\
K'\ar[r]_{\psi} \ar[rru]|!{[rr];[u]}\hole   & K\ar[ru]&
 }
\end{aligned}
\end{equation}

\subsection{Relative homology functors.}  For each $d\in \Z$, \Cech homology determines a functor  
\bear\label{1.Cechfunctor}
\cHH_d:{\mathcal K}_\P\to \Ab
\eear
 with values in  the category $\Ab$ of abelian groups by $\phi\mapsto \cHH_d(\ov\M_\phi)$. 

\begin{defn}
\label{2.RHFdefn}
 A lift  of \eqref{1.Cechfunctor} to a functor 
\bear\label{1.functor}
\mu:{\mathcal K}_\P\to \Ab^* 
\eear
into the category of  abelian groups  with a distinguished element is called  a  {\em  relative homology functor}   of degree $d$  associated to $\ov\pi:\ov\M\to\P$.
\end{defn}

Thus a  relative homology  functor assigns to each  continuous map $\phi:K\to\P$ from a  non-empty, compact,  path-connected Hausdorff space $K$ an element  
$$
\mu(\phi)\in \cHH_d(\ov\M_\phi)
$$
such that each morphism  \eqref{2.diagram1} induces an equality
\bear \label{1.4}
\wh\psi_*[\mu(\phi')] =\mu(\phi).
\eear
Two special cases are especially important:
\begin{enumerate}[(i)]
\item For each $p\in \P$ , we can take $\phi$ to be the inclusion of $p$ into $\P$.  Then $\ov\M_\phi$ is identified with the fiber 
$\ov\M_p$ over $p$, and we  obtain a class 
$$
\mu(p)\in \cHH_d(\ov\M_p).
$$
\item   For each  path $\gamma:[0,1]\to \P$ with endpoints $p$ and $q$, let $\iota_0$ (resp. $\iota_1$) be the inclusion $p\mapsto 0\in[0,1]$ (resp. $q\mapsto 1\in[0,1]$).  Applying 
 \eqref{1.4} first with $\psi=\iota_0$ and $\phi=\gamma$, then with $\psi=\iota_1$ and $\phi=\gamma$, yields  the consistency condition  
  \bear\label{2.consistent}
(\iota_0)_*\mu(p)=  (\iota_1)_*\mu(q) \quad \mbox{in} \quad \cHH_d( \ov{\M}_{\gamma}).
\eear

\end{enumerate}

\medskip

The importance of these two cases is reflected in the following extension and uniqueness result.
The proof is an application of the \Cech continuity property \eqref{1.Cech.ContinuityProperty}.

\begin{prop}\label{extLemma}
\label{2.ExtensionLemma}
 Let $\ov\pi:\ov\M\ra \P$ be a proper continuous map from a Hausdorff space to a locally path-connected metrizable space. 
\begin{enumerate}\setlength{\itemsep}{4pt}
\item[(a)] (Uniqueness)   If two relative homology functors associated to $\ov\pi$  are equal on a dense set of points $p\in\P$, then they are equal.
\item[(b)]  (Extension)   Suppose that there is a  dense subset $\P^*$ of $\P$, and an assignment
\bear\label{lemma3.4eq1}
p \mapsto \mu(p)\in \cHH_d(\ov{\M}_p)
\eear
defined for $p\in \P^*$   such that,  for any $p, q\in \P^*$,  \eqref{2.consistent} holds for each  path $\gamma:[0,1]\to\P$   in a $C^0$ dense subset of the space of paths in $\P$ from $p$ to  $q\in\P^*$. Then \eqref{lemma3.4eq1} extends uniquely to a relative homology functor.
\end{enumerate}

 \end{prop}

\begin{proof} 
(a)  Relative homology functors satisfy \eqref{1.4}, and hence \eqref{2.consistent}.  Statement (a) therefore follows from the uniqueness in statement (b).

(b) After fixing a metric on $\P$, the hypotheses of (b)  are the same as those of  Extension Lemma~3.4 of \cite{IPThin}, which shows that \eqref{lemma3.4eq1} extends uniquely  to an assignment, still denoted $p \mapsto\mu(p)$,  defined for all $p\in\P$ and satisfying \eqref{2.consistent} for all paths $\gamma$ with endpoints $p,q\in\P$.    Then,  given a map $\phi:K\to\P$ in $\KP$, choose a point $p\in K$ and set  
\bear\label{1.defmuphi}  
\mu(\phi)=\wh\iota_* \mu(p) \in \cHH_d(\ov{\M}_\phi), 
\eear
  where $\wh\iota$ is the inclusion $\ov\M_p\hookrightarrow \ov\M_\phi$.  Given another point $q\in K$, choose  a path $\sigma:[0,1]\to K$ from $p$ to $q$.  Applying
\eqref{2.consistent}  to the path $\gamma=\phi\circ \sigma$ and pushing forward  in homology by  the map $\wh{\si}: \ov\M_\gamma \ra \ov\M_\phi$ shows that the class \eqref{1.defmuphi}  is  independent of  the choice of $p$.    With this established, \eqref{1.4} follows from \eqref{1.defmuphi}   by applying $\wh\psi_*$. 

 Finally, this extension is unique:   two relative homology functors $\mu$ and $\mu'$  that agree on all points $p\in\P^*$ must also agree for all $p\in\P$ by the uniqueness of \cite[Lemma~3.4]{IPThin}. But then they  agree for all $\phi$ in $\KP$:   applying  \eqref{1.4}, for both $\mu$ and $\mu'$,  to the inclusion of any point $p \hookrightarrow K$ shows that $\mu(\phi) =  \wh\iota_* \mu(p) =\wh\iota_* \mu'(p) = \mu'(\phi)$.
\end{proof}

\bigskip

We next note four functorial constructions.   All  four start with a   relative homology functor $\mu$ of degree $d$ on a
 family $\ov\pi:\ov\M\to \P$ as in \eqref{1.8diagram}.

\subsection{Pullbacks.}    For each continuous map $\sigma:{\cal Q}\to \P$
from a space ${\cal Q}$, there is an associated pullback family and a commutative square
\bear\label{1.subfamily}
\xymatrix@=5mm{
\ov{\M}_\sigma  \ar[d]_{\ov\pi_\sigma} \ar[r]_{\widehat{\sigma}} & \ov{\M} \ar[d]^{\ov\pi}   \\
{\cal Q}\ar[r]_\sigma & \P
} 
\eear
where, one can check, $\ov\pi_\sigma$ is proper.   Then there is an induced  pullback relative homology functor $\sigma^*\mu$:  for each  $\phi:K\to{\cal Q}$ in ${\mathcal{K_{\cal Q}}}$, the composition $\sigma\circ \phi$ is in $\KP$, and we define
$$
(\sigma^*\mu)(\phi)\ =\ \mu(\sigma\circ \phi) 
%\ \in\  \cHH_d(\ov\M_{\sigma\circ\phi}) = \cHH_d((\ov\M_{\sigma})_\phi).
$$
in the \Cech homology, using the identification  of $(\ov\M_{\sigma})_\phi$ with $\ov\M_{\sigma\circ\phi}$, defined  by the projection $(\ov\M_{\sigma})_\phi\ra \ov\M_{\sigma\circ\phi}$.
 Thus defined, $\sigma^*\mu$ satisfies \eqref{1.4}, so is a relative homology functor of degree $d$ associated to $\ov\pi_\sigma$ .

\subsection{ Pushforwards and Invariants.}   Given a continuous map
$$
\ov{f}:\ov\M\to Z
$$
to a Hausdorff space $Z$,  one  obtains numerical invariants as follows.

\begin{cor}
\label{1.invtcor}
For $p\in\P$ and $\beta\in \cHH^*(Z;\Z)$,  the class
$$
\ov{f}_*\mu(p)\in \cHH_d(Z;\Z)
$$
and the function
\bear\label{1.invts}
I_\beta(p)\ =\ \left\langle\ov{f}_* \mu(p), \;  \beta \right\rangle
\eear
are independent of $p$ on path-connected components of $\P$. 
\end{cor}

\begin{proof}
Any pair $p,q$ of points in the same path-connected component of $\P$ are the endpoints of a path $\gamma:[0,1]\to \P$.  By \eqref{2.consistent}, $\ov{f}_*\mu(p)$ is equal to $\ov{f}_*\mu(q)$ in the image of $\ov{f}_*\gamma_*: \cHH_d(\ov\M_\gamma)\to  \cHH_d(Z)$, and hence $I_\beta(p)=I_\beta(q)$.
\end{proof}

\subsection{Cap Products.} Suppose that  \eqref{1.8diagram} is such that   there is a  cap product 
\best 
\cap: \cHH_d(X)  \otimes \cHH^k(X)\ra \cHH_{d-k}(X) 
\eest
defined on a category of spaces $X$ that includes all $\ov\M_\phi$ for $\phi\in{\cal K}_\P$, and which has the naturality property $  f_*a \cap \xi =f_*(a\cap f^*\xi)$ for all maps $f:X\to Y$ in that category.   Then, 
given a relative homology functor $\mu$, each cohomology element $\alpha\in \cHH^*(\ov\M)$ determines a relative homology functor $\mu \cap \alpha$ defined by
\bear\label{1.20}
(\mu \cap \alpha)(\phi)\,=\, \mu(\phi)\cap \wh\phi^*\alpha,
\eear
for each $\phi\in \KP$.   Indeed, for maps $\phi'=\phi\circ\psi$ and $\wh{\phi}'=\wh\phi\circ \wh\psi$ as in \eqref{2.diagram1} and \eqref{2.Cat}(b), the naturality of $\cap$  and formula  \eqref{1.4} for $\mu$ imply that
\begin{align}\label{1.capcalc}
\wh\psi_*\big[(\mu\cap \alpha)(\phi')\big]\, =\,    
 \wh\psi_*\big[\mu(\phi')  \cap \wh\psi^*(\wh\phi^*\alpha) \big]\, =\, 
\big[\wh\psi_*\mu(\phi') \big]\cap (\wh\phi^*\alpha) 
\, =\, (\mu\cap \alpha)(\phi).
\end{align}
Thus $\mu\cap\alpha$ satisfies  \eqref{1.4}, so is a relative homology functor.

\medskip

Sections~\ref{section5} and \ref{section6} describe how the invariants \eqref{1.invts}  and the cap products  \eqref{1.20} are related  to intersection numbers.

\subsection{Relative fundamental classes.}\label{sS.rfc}   Following Definition~3.1 of  \cite{IPThin}, we now impose additional structure on the map $\ov\pi:\ov\M\ra \P$.  In particular, we assume that  $\P$ is locally path-connected and metrizable, and is a Baire space.

\begin{defn}\label{1.def.rtf}
We say that a family   $\ov\pi:\ov\M\ra \P$ as in \eqref{1.8diagram}  is a {\em relatively thin family} if  $\P$ is  a  locally path-connected   metrizable  Baire space, and  there  is a number $d$ and  a second category set $\P^*\subseteq \P$  such that for each $p,q\in\P^*$
\begin{enumerate}\setlength{\itemsep}{4pt}
\item[(a)]   $\ov\M_p$  is a thin compactification of the $d$-dimensional oriented manifold $\M_p=\pi^{-1}(p)$.

\item[(b)]  The space of continuous paths in $\P$ from   $p$ to $q$ contains a $C^0$-dense subset of paths $\gamma$ for which the fiber product  $\ov\M_\gamma$ 
 is  a thin cobordism from $\ov\M_p$ to $\ov\M_q$.
\end{enumerate}
\end{defn}
These conditions ensure that the  generic fiber $\ov\M_p$ of $\ov{\pi}$  has a fundamental class as in \eqref{1.Cechfc}.   This leads to the  notion of  a relative fundamental class, which is our key object of study.

\begin{defn}
\label{1.defnVFC}
 A {\em relative fundamental class} \mbox{\rm (RFC)}   of a relatively thin family   is a relative homology functor $\mu$  of degree $d$ such that
\best\label{1.normalization}
\mu(p)=[\ov{\M}_p]
\eest
for each $p$ in the set $\P^*$.  We will often  write $\mu(\phi)$ as   $[\ov\M_\phi]^\rfc$.
 \end{defn}
 
This is equivalent to  Definition~4.1 of \cite{IPThin}:  one direction is clear, the other follows immediately from Proposition~\ref{extLemma}(b).

 \medskip

Note that  a RFC is not a single  element of the  homology of some space.    Rather, 
it is a functor that, as in  \eqref{1.functor},  assigns to each  continuous map $\phi: K\to \P$  from a non-empty, compact  path-connected space $K$,  a \Cech homology class 
$$
\mu(\phi) = [\ov\M_\phi]^\rfc \in \cHH_d(\ov\M_\phi)
$$
that  satisfies  a naturality axiom and a normalization axiom:
\begin{itemize} \setlength\itemsep{10pt}

\item[{\bf A1.}] Every triangle \eqref{2.diagram1} of  continuous maps, where $K$ and $L$  are  nonempty, compact, and path-connected,  induces an equality  
\bear\label{1.A1}
\widehat \psi_*[\ov{\M}_{\phi'}]^\rfc= [\ov{\M}_\phi]^\rfc. 
\eear

\item[{\bf A2.}] For each  $p\in \P^*$, 
\bear\label{1.A2}
\mbox{  $[\ov{\M}_p]^\rfc$ is the fundamental class $[\ov{\M}_p]$.   }
\eear
\end{itemize}

  \bigskip

Finally, we note that relative fundamental classes are natural  under certain changes of the parameter space $\P$ (cf. \cite[Section 6]{IPThin}).    A {\em morphism} between   relatively thin  families  $\ov\pi$ and $\ov\pi'$ is a  diagram of continuous maps 
   \bear\label{2.QPdiagram}
\xymatrix@=5mm{
\ov{\M}  \ar[d]_{\ov\pi} \ar[r]_{\widehat f} & \ov{\N} \ar[d]^{\ov\pi'}   \\
\P\ar[r]_f & {\cal Q}.
} 
\eear
We say that  a morphism \eqref{2.QPdiagram}  is {\em generically degree 1} if there exist  second category subsets   $\P^*$ of $\P$ and ${\cal Q}^*$ of ${\cal Q}$ satisfying 
the condition  of Definition~\ref{1.def.rtf}(a), with $f(\P^*) \subseteq {\cal Q}^*$, and such   for each $p\in\P^*$\!,  $\wh f $ restricts to   a degree~1 map  
\bear\label{2.MpNp}
\widehat{f}_p:\M_p\to \N_{f(p)}.
\eear

\begin{lemma}
\label{basechangelemma}
If  a morphism \eqref{2.QPdiagram} of relatively thin families is  generically degree 1, then 
\bear\label{2.pullback}
\mu^\rfc\ =\ f^*\nu^\rfc.
\eear
 \end{lemma}

\begin{proof}
   If $p\in\P^*$ and $q=f(p)\in {\cal Q}^*$, then 
$$
\mu^\rfc(p) =[\ov\M_p] \qquad \mbox{and}\qquad (f^*\nu^\rfc)(p) = \nu^\rfc(q)=[\ov\N_q].
$$
But  $(\widehat f_p)_*[\ov\M_p]=[\ov\N_q]$  by assumption \eqref{2.MpNp}.   Thus   $f^*\nu^\rfc(p)    = \mu^\rfc(p)$ for a dense set of points $p$ in $\P$.   The lemma then follows by  Proposition~\ref{2.ExtensionLemma}(a).  
\end{proof}

The pullback property \eqref{2.pullback} is not true for general morphisms \eqref{2.QPdiagram};  the two  relative homology functors  may not even have the same dimension.  In particular,   if one restricts the space $\P$ of parameters to a submanifold of $\P$, the  relative fundamental classes need not correspond.   Examples of this phenomenon are given in Section~6 of  \cite{IPThin}.

 \vspace{13pt}

The constructions in this section required two homology theories.  We
first used the {\em exactness}  of Steenrod homology to extend  fundamental classes from manifolds to their thin compactifications; we  then passed to \Cech homology and used its {\em continuity} property to extend fundamental classes to all  fibers in  a family.   Unfortunately, on the category of compact pairs with  $\Z$ coefficients,  no homology theory is both exact and continuous \cite[\S X.4]{ES}.

  On the other hand, with coefficients in  $\Q$,  Steenrod and \Cech homology are naturally  isomorphic on the category $\AC$ of compact spaces and continuous maps, giving a single theory   that is both exact and continuous (and essentially unique -- see  \S A.2 of the appendix).  The constructions of this section then produce an RFC in   rational Steenrod homology.  This approach avoids \Cech homology, at the expense of losing track of whether   the invariants \eqref{1.invts} are integers.

%%%%%%%%%%%%%%%%%%%%%%%%%%%%%%%%%%%%%%%%%%%%%%%%%%%%%%%%%%%%
%%%%%%%%%%%%%%%%%  Section 2  %%%%%%%%%%%%%%%%%%%%%%%%%%%%%%
%%%%%%%%%%%%%%%%%%%%%%%%%%%%%%%%%%%%%%%%%%%%%%%%%%%%%%%%%%%%%%
\setcounter{equation}{0}
\section{Fredholm Families and  relative fundamental classes} 
\label{section2}
\medskip

 We now  define and focus attention on a class of families  \eqref{1.8diagram} where the structures of Section~1 arise naturally via the Sard-Smale theorem.   The motivating examples occur in   gauge theories,  where one has universal moduli spaces $\M$ which are Banach manifolds  and  have compactifications $\ov\M$ with natural maps to a manifold $\P$ of parameters.  The definitions and results of this section codify the relevant structure of such moduli spaces which is needed to ensure the existence of a    relative fundamental class.

\smallskip 

In this and later sections, the term ``Banach manifold'' means a metrizable separable  
Banach manifold, finite or infinite dimensional. Such manifolds are second countable and
paracompact  (metrizability is needed  to apply the  dimension theory in Appendix~B).  By a {\em Fredholm family} we mean a   Fredholm map
\bear
\label{2.1}
\xymatrix@R=6mm{
\M \ar[d]^{\pi}\\
\P
}
\eear
between $C^l$   Banach manifolds, finite or infinite dimensional, which we again regard as a family of spaces (the fibers of $\pi$) parameterized by  $\P$.    Such a map $\pi$  has an associated Fredholm index $d$, and we assume that 
\bear\label{2.lbound}
 l > \max(d+1,0).
\eear

We also assume that \eqref{2.1}  comes equipped with two additional structures:
\begin{itemize} \setlength\itemsep{4pt}
\item  A relative orientation specified by a nowhere zero section of the relative determinant bundle.  A relative orientation on $\M$ induces an orientation on each regular fiber $\M_p$.
\item    A {\em (metrizable) relative compactification}, meaning a metrizable  space  $\overline{\M}$ together with  a commutative diagram 
\bear
\label{1.5diagram}
\xymatrix{ 
\M\ar[r]^\phi  \ar[d]_\pi& 
\ov\M \ar[dl]^{\ov\pi}\\
\P
}
\eear
 where $\phi$ is an  inclusion of $\M$ as an open subset, and  $\ov\pi:\ov\M\to \P$ is continuous and proper.
\end{itemize}

\medskip
 
  We call the set $\SS = \ov\M\Setminus \M$ the {\em singular locus} of $\ov\M$,  so   $\ov{\M}$ is the disjoint union
$$
\overline\M \ =\  \M\ \cup\  \SS.
$$
The following definition  generalizes  and supersedes the notion of a ``Fredholm-stratified thin compactification'' defined in   \cite{IPThin}, and casts it in terms of three easily-verifiable conditions.  (Definition~5.2 in \cite{IPThin} is a special case in which each $\phi_\alpha$ is an inclusion and the images $\phi_\alpha(\Sa)$ are disjoint.)

\begin{defn}
\label{1.Def1.1}
Fix a relatively oriented Fredholm family  \eqref{2.1}  of index~$d$.   A  {\em  Fredholm thin compactification} of $\M$ is a metrizable relative  compactification  as in \eqref{1.5diagram}, together with   a countable set 
 $\A$ and, for each $\alpha\in\A$,   a diagram
\bear\label{1.MNP}
 \xymatrix{
\Sa \ar[r]^{\phi_\alpha} \ar[dr]_{\pi_\alpha}& \ov{\M} \ar[d]^{\ov\pi} \\
& \P   
}\eear
such that
\begin{enumerate}\setlength{\itemsep}{4pt}
\item $\pi_\alpha:\Sa\to \P$ is a  Fredholm family of index $d_\alpha \le d-2$.
\item Each $\phi_\alpha$ is either (i) continuous and locally injective, or (ii) locally Lipschitz.
\item  $\{\phi_\alpha(\Sa)\, |\, \alpha\in\A\}$ cover   $\SS$.
\end{enumerate}
\end{defn}
These conditions imply that the fibers  $\ov\M_p$ and the cobordisms $\ov\M_\gamma$ are compact metrizable spaces. 
 
\medskip

In this context,  the Sard-Smale theorem yields a crucial fact:

 \begin{lemma}
 \label{lemma1.2}
A  Fredholm thin compactification \eqref{1.MNP}  is a relatively thin family.
\end{lemma}

\begin{proof}  Because $\P$ is a separable metrizable Banach manifold, it is locally path-connected,  and is a Baire space.  It remains to verify conditions (a) and (b) of Definition~\ref{1.def.rtf}.

(a)\ The Sard-Smale theorem, applied to $\pi$ and to each $\pi_\al$, $\alpha\in \A$, 
shows that there are  second category subsets $\P_0$ and $\P_\alpha$ of $\P$ such that (i) the fiber $\M_p$ over each $p\in\P_0$ is a  $C^1$ manifold  of dimension  $d$ (with orientation induced from  the relative orientation),  and (ii) 
 the fiber $\Sap$  of $\pi_\al$ over each $p\in\P_\alpha$ is a  $C^1$ manifold of dimension  $d_\alpha\le d-2$. Then
 \bear\label{2.regularvalues}
 \P^*= \P_0\cap \bigcap \P_\alpha
 \eear
  is also a second category subset of $\P$.   We  call elements of $\P^*$ the {\em regular values of $\ov\pi$}.

     For each $p\in\P^*$,   $\SS_p=   \ov\M_p\Setminus\M_p$ is a closed, hence compact, subset of $\ov\M_p$.  By Definition~\ref{1.Def1.1}, $\SS_p$ is covered by the sets $\phi_\al(\Sap)$.      Each  $\SS_{\alpha, p}$  is a  finite-dimensional submanifold of the (second countable metrizable) Banach manifold $\SS_\alpha$, so is $\sigma$-compact.    Assumption~(2) and Lemma~\ref{LemmaA2}  then show that
$$
\sHH_k(\SS_p)=0 \ \quad \forall\, k> d-2. 
 $$
Thus $\ov{\M}_p$ is a metrizable thin compactification of $\M_p$.

  (b)  Now fix  $p, q\in \P^*$ and $l$ satisfying \eqref{2.lbound}.  Because any continuous path is the $C^0$ limit of $C^l$ paths,  it suffices to show that  (b) holds for  a dense subset of the space $\P^l(p,q)$ of $C^l$ paths $\gamma:[0,1]\to \P$  from $p$ to $q$. Furthermore,  noting that $p$ and $q$ are regular values of $\pi$,  the Sard-Smale theorem shows that
 the set of $\gamma\in \P^l(p,q)$ that are transverse to $\pi$ is open and dense, and the same is true with $\pi$ replaced by  $\pi_\alpha$ for each $\alpha$. The intersection of these sets is a dense subset of 
  paths $\gamma\in\P^l(p,q)$ for which (i) the fiber product  \eqref{1.M_K}  is a $(d+1)$-dimensional oriented manifold 
  \best
\M_\gamma
\eest 
 whose boundary is canonically identified with $\M_p\sqcup\M_q$, and (ii)  each fiber product $({\cal S}_{\alpha})_\gamma$  is a  manifold of dimension $d_\al+1\le d-1$ with boundary 
${\cal S}_{\alpha, {\bd \gamma}}={\cal S}_{\alpha, p}\sqcup  {\cal S}_{\alpha, q}$ and with maps
\bear\label{2.converS}
 \mathrm{id}\ti \phi_\al : ({\cal S}_{\alpha})_\gamma \to\ov\M_\gamma
\eear
 also satisfying property (2) of Definition~\ref{1.Def1.1}. By Definition~\ref{1.Def1.1}(3)  and the commutativity of  diagram~\eqref{1.MNP}, the images  of the maps \eqref{2.converS} cover the  singular locus of $\ov\M_\gamma$ and the images of their restrictions to ${\cal S}_{\alpha, {\bd \gamma}}$  cover the singular locus of the boundary  $\ov\M_{\bd\gamma}= \ov\M_p \sqcup\ov\M_q$. We can then similarly apply Lemma~\ref{LemmaA2} to conclude that the singular locus of $\ov\M_\gamma$ satisfies correct 
properties to be a thinly compactified cobordism from  $\ov\M_p$ to $\ov\M_q$. 
 \end{proof} 

\bigskip

 Combining Lemma~\ref{lemma1.2} and Proposition~\ref{2.ExtensionLemma} leads to our first main theorem.

 \medskip

\begin{theorem}
\label{theorem1.2}
 A  Fredholm  thin compactification  $\ov\pi:\ov{\M}\to \P$ admits  a unique relative fundamental class.
 \end{theorem}

\begin{proof}
First apply Lemma~\ref{lemma1.2}, noting that the set $\P^*$ in \eqref{2.regularvalues} of regular values of $\ov\pi$ is dense in $\P$ by the Baire Category theorem.  Then, as noted after Definition~\ref{1.def.rtf}, $\ov\M_p$ has a \Cech fundamental class $[\ov\M_p]$ for each $p\in\P^*$.  Define an assignment \eqref{lemma3.4eq1} by setting 
\bear\label{2.last}
\mu(p) =  [\ov\M_p]\in\cHH_d(\ov\M_p)
\eear
for $p\in\P^*$.  For any $q\in\P^*$,  Lemma~\ref{lemma1.2} also  shows that there is a dense set of paths $\gamma$ from $p$ to $q$, each with an associated thin cobordism $\ov\M_\gamma$.   Applying \eqref{1.3} to this cobordism shows that  the consistency condition  \eqref{2.consistent} for these paths $\gamma$.

Proposition~\ref{2.ExtensionLemma} now applies ($\P$ is a separable Banach manifold, so is locally connected and metrizable).  Thus the assignment $p\mapsto \mu(p)$ extends uniquely to a relative homology functor that satisfies  \eqref{2.last}.
 \end{proof}
 
 \medskip

\begin{rem}\label{ex2.4}
  Suppose that a family as  in Theorem~\ref{theorem1.2}  contains a {\em complex algebraic subfamily}, meaning that there is a diagram \eqref{1.subfamily} where  $\sigma:{\cal Q}\to \P$ is an inclusion and $\ov\pi_\sigma$ is a proper complex algebraic map between varieties.  As in \S 1.4,  the RFC on $\ov\M$ pulls back to a relative homology functor on $\ov\M_\sigma$, so for 
  each proper algebraic map $\phi:K\to {\cal Q}$, one obtains a \Cech homology class 
  \bear\label{2.Algebraic}
  [\ov\M_{\sigma\circ \phi}]^{\rfc}.
  \eear
 Because $\ov\M_{\sigma\circ \phi}$ is compact and locally contractible, its Steenrod, \Cech, and  Borel-Moore homologies are isomorphic with $\Z$ coefficients (cf. \S A.3).  Thus \eqref{2.Algebraic} can be regarded as a class in  Borel-Moore homology, which is more commonly used  by algebraic geometers. 
    
Note that the class \eqref{2.Algebraic}   is defined even if $\ov\M_{\sigma\circ \phi}$ has no regular fibers, and it   has the naturality property \eqref{1.A1} with respect to proper algebraic base changes with path-connected base. However, its   existence and uniqueness depend on the existence of a non-algebraic object:  the Fredholm thin compactification of Theorem~\ref{theorem1.2}.
 \end{rem}

%%%%%%%%%%%%%%%%%%%%%%%%%%%%%%%%%%%%%%%%%%%%%%%%%%%%%%%%%%%%
%%%%%%%%%%%%%%%%%  Section 3  %%%%%%%%%%%%%%%%%%%%%%%%%%%%%%
%%%%%%%%%%%%%%%%%%%%%%%%%%%%%%%%%%%%%%%%%%%%%%%%%%%%%%%%%%%%%%
\setcounter{equation}{0}
\section{Proper maps and pseudo-cycles } 
\label{section3}
\medskip

In  geometric topology,  intersection invariants can be defined using  pseudo-cycles.
This section gives a general definition of pseudo-cycle, relates it to the   definition   used in  \cite{ms2}, and describes how pseudo-cycle  classes can be realized by pushing forward the fundamental class of a thinly compactified manifold.  As in \cite{IPThin} and \cite{sw}, we  work  in Steenrod homology with coefficients in $\Z$ or $\Q$,  although the results and proofs hold equally well in Borel-Moore homology (see  Appendix~\ref{sectionA} and Lemma~\ref{LemmaB1}). The  connection with intersection numbers  is made in the  next section.

\subsection{Proper maps.}   In Steenrod and Borel-Moore homology, fundamental classes push forward only under {\em proper} continuous maps.  Thus it is helpful to describe two ways (Lemmas~\ref{Lemma3.1} and \ref{Lemma4.1B} )  that a continuous map
\bear\label{3.0}
 f:M\to Z
\eear
 between Hausdorff spaces can be modified to produce a proper map.    Recall that the {\em Omega limit set} of $f$ is defined  to be 
$$
\Omega_f\ =\  \bigcap_{K\subseteq M} \ \overline{f(M\Setminus K)},
$$
where the bar denotes closure and the intersection is over all compact sets $K\subseteq M$.

 \begin{lemma}\label{Lemma3.1} 
For any subset $A$ of  $Z\Setminus \Omega_f$, the  restriction of $f$ to $M_A=f^{-1}(A)$  is a proper map
\bear\label{4.MU}
f_A: M_A\to A.
\eear
 \end{lemma}
\begin{proof}  Fix a compact set $C\subseteq A$. We must show that the closed set $f_A^{-1}(C)$ is compact.  This is true if 
 $f_A^{-1}(C)$ lies in some compact $K\subseteq M$. Otherwise, for each compact $K\subseteq M$,  the set $B_K=C\cap  \overline{f(M\Setminus K)}$ is a non-empty, closed -- hence compact --  subset of $C$. But then  $\bigcap_K B_K =  C\cap \Omega_f$ is non-empty, contradicting the fact that $C\cap \Omega_f \subseteq A\cap \Omega_f=\emptyset$.   
 \end{proof}

\medskip
 
 Instead of restricting $f$, we can extend it.  
 
 \begin{lemma}
\label{Lemma4.1B}
Suppose that  $M$ is a subset of a compact space $\ov{M}$ with closed complement $S=\oM\Setminus M$.  If  $\ov{f}:\ov{M}\to Z$ is  a continuous extension of \eqref{3.0}, then $\ov{f}$ is proper and  $\Omega_f\subseteq \ov{f}(S)$.
\end{lemma}

\begin{proof}
The first conclusion is evident because a continuous map from a compact space to a Hausdorff space is proper.  For the second,   suppose
 by contradiction that there is  a point $y\in\Omega_f$ that is not in $\ov{f}(S)$.  The hypothesis implies that $S$ is compact and hence so is  $\ov{f}(S)$.  Hence we can find disjoint open neighborhoods $U$ of $y$ and $V$ of $\ov{f}(S)$. Then $K=\oM\Setminus \ov{f}^{\; -1}(V)$ is a compact  subset of $M$,  and $f(M\Setminus K) \subseteq \ov{f}(\oM\Setminus K) \subseteq V$.    But this implies that $y\notin \ov{f(M\Setminus K)}$, so $y\notin\Omega_f$, giving a contradiction.  
 \end{proof}

\subsection{Pseudo-cycles.}   A $d$-dimensional  {\em pseudo-cycle}  is a continuous map 
\bear\label{3.3}
 f:M\to Z
\eear
from an oriented  $d$-dimensional topological manifold $M$ to a  locally compact metric space $Z$  
such that $f(M)$  has compact closure and 
\bear\label{pseudocyclecondition}
\dim \Omega_f \ \le\ d-2,
\eear
where $\dim$ denotes the  Lebesgue covering dimension. These conditions imply that $\Omega_f$ is compact and hence, by Lemma~\ref{LemmaB1},
\bear\label{3.om.s}
 \sHH_k(\Omega_f; \Z)= 0\qquad \mbox{ for all $k\ge d-1$.}
 \eear
 This notion of pseudo-cycle generalizes the one used in \cite{ms2}, cf. Lemma~\ref{Lemma4.3A} below.  Note that if $f$ is proper, then $\Omega_f=\emptyset$, hence $f$ is a pseudo-cycle.

  Two such pseudo-cycles $f_i:M_i \ra Z$, $i=1,2$, are called {\em cobordant} if there exists a $(d+1)$-dimensional oriented manifold $W$ with boundary $\bd W= M_2\sqcup (-M_1)$ and a map $F:W \ra Y$ such that 
\bear\label{3.cob}
F|_{M_1}= f_1, \quad F|_{M_2}= f_2, \quad \dim \Omega_F\ \le\ d-1,  \mbox{ and $\ov F(W)$ is compact.}
\eear
\medskip

\smallskip

The following result is due to  M.~Schwarz \cite[Theorem~3.1]{sw}. 
\begin{lemma}\label{pseudocycleclass}
A pseudo-cycle \eqref{3.3}  determines a Steenrod class 
\bear\label{4.psClass}
[f] \in  \sHH_d(X; \Z),
\eear
where $X=\ov{f(M)}\subseteq Z$. It is defined by formula \eqref{6.pcclass} below.  
\end{lemma}
\begin{proof} 
 By Lemma~\ref{Lemma3.1}, $f$ restricts to a proper map $f^o:M^o \to  X\Setminus \Omega_f$ whose domain
$$
M^o= f^{-1}(Z\Setminus \Omega_f)
$$
 is an open subset of the oriented $d$-manifold $M$.  Hence $M^o$ has a  fundamental class which satisfies
$$
[M^o]=\rho_o[M],
$$
(see \eqref{5.2MN}), where $\rho_o=\rho_{M^o}$.

The long exact sequence \eqref{1.LES}, together with \eqref{3.om.s}, shows that  the restriction to   $U=Z\Setminus \Omega_f$ induces an isomorphism 
\bear\label{6.YYFrseq}
\begin{tikzcd}
0 \arrow{r} & \sHH_d(X)  \ar{r}{\rho}[swap]{\cong}&   \sHH_d(X\Setminus \Omega_f)  \arrow{r} & 0.
\end{tikzcd}
\eear
Thus  the image of $[M]$  under the composition 
\bear\label{M->f}
\begin{tikzcd}
\sHH_d(M)  \arrow{r}{\rho_o}& \sHH_d(M^o)   \ar{r} {f^o_*} &\sHH_d(X\Setminus \Omega_f) 
 \ar{r}{\rho^{-1}}[swap]{\cong} & \sHH_d(X). 
\end{tikzcd}
\eear
determines a class 
  \bear\label{6.pcclass}
[f] = \rho^{-1}f^o_*[M^o] \ =\ \rho^{-1}f^o_*\rho_o[M] \  \in \  \sHH_d( X; \Z). 
\eear
\end{proof}

If $f$ is a  proper map then its pseudo-cycle class \eqref{4.psClass} is simply
$$
[f]= f_*[M].
$$
 More generally,  if $f$ is a   pseudo-cycle  and $U$ is an open subset  of  $Z\Setminus \Omega_f$, 
  $f_U:M_U\ra U$  is proper by Lemma~\ref{Lemma3.1}, and 
 \bear\label{Lemma3.4eq}
\rho_U[f]=[f_U]\quad \mbox{ in }\sHH_d(U).
\eear
This  follows from the commutative diagram obtained by restricting   \eqref{M->f} over $U$:
\bear\label{M->f.U}
\begin{tikzcd}
\sHH_d(M)\arrow{d}{\rho_{M_U}}  \arrow{r}{\rho_o}& \sHH_d(M^o) \arrow{d}{\rho_{M_U}}   \ar{r} {f^o_*} &\sHH_d(X\Setminus \Omega_f) 
 \ar{r}{\rho^{-1}}[swap]{\cong} \arrow{d}{\rho_U} & \sHH_d(X)\arrow{d}{\rho_U} 
 \\
 \sHH_d(M_U)  \arrow{r}{\rho_o=id}[swap]{=} & \sHH_d(M_U)   \ar{r} {f_{U*}} &\sHH_d(U) 
 \ar{r}{\rho^{-1}=id }[swap]{=} & \sHH_d(U). 
\end{tikzcd}
\eear

\medskip

\begin{lemma}\label{L.indep.cob}If $F:W \ra Z$ is a cobordism between  $d$-dimensional pseudo-cycles $f_1$ and $f_2$, then 
\bear\label{f0=f1}
[f_1]=[f_2]\quad \mbox{ in }\sHH_d(Y),
\eear
 where $Y= \ov {F(W)}$ is the closure of the image of the cobordism. 
\end{lemma}
\begin{proof} As in \eqref{3.om.s}, the assumptions \eqref{3.cob} imply that $\sHH_d(\Omega_F) =0$. Then \eqref{1.LES} shows that the restriction to the open subset $Y^o=Y\Setminus \Omega_F$ of $Y$ is an  injection:
$$   
\begin{tikzcd}
0 =\sHH_d(\Omega_F)\arrow{r} & \sHH_d(Y)  \ar{r}{\rho_{Y^o}}&   \sHH_d(Y\Setminus \Omega_F). 
\end{tikzcd}
$$
Therefore it suffices to prove the equality of the restrictions of \eqref{f0=f1} to $Y^o$. 

By Lemma~\ref{Lemma3.1},  the restriction $F^o:W^o\ra Y^o$ of $F$ over $Y^o$ is proper. But 
\best
W^o= F^{-1}(Y^o)= W\Setminus F^{-1}(\Omega_F)
\eest
  is an open subset of $W$, so is a manifold with boundary 
$\bd W^o =\bd W\cap W^o= \bd W \Setminus F^{-1}(\Omega_F)$.
Moreover, as in the proof of \cite[Lemma~2.10]{IPThin}, under the long exact sequence sequence 
\best
\label{}
\xymatrix{
\sHH_{d+1}(W^o)\ar[r]^{\rho\quad}& \sHH_{d+1}(W^o\Setminus\partial W^o) 
\ar[r]^{\quad \bd} & \sHH_{d}(\partial W^o) \ar[r]^{\iota_*} &\sHH_{d}(W^o)
}
\eest
we have  $\partial [W^o\Setminus \partial W^o]=[\partial W^o]$, and hence
 $\iota_* [\partial W^o] =0$.   Pushing forward by the proper map $F^o:W^o\ra Y^o$ gives
\bear\label{F.bd=0}
F^o_*[\partial W^o] =0  \quad \mbox{ in } \sHH_{d}(Y^o). 
\eear
Because  $\bd W= M_2\sqcup (-M_1)$, we have 
\best
\partial W^o= \bd W \Setminus F^{-1}(\Omega_F)= M_2^o \sqcup (-M_1^o),
\eest
where each
$
M_i^\circ = f_i^{-1}(Y^o)=M_i \Setminus f_i^{-1}(\Omega_F)
$
is an open subset of $M_i$ and therefore a submanifold. But the restriction $f_i^o$ of $f_i$ over $Y^o$ is equal to the restriction of $F^o$ to $M_i^o$, and therefore
\best
F^o_*[\partial W^o]=F^o_* [M_2^o]-F^o_*[M_1^o]= [f_2^o]-[f_1^o] = \rho_{Y^o} \big([f_2]- [f_1] \big). 
\eest
Combining this with  \eqref{F.bd=0} shows that \eqref{f0=f1} holds after restriction to $Y^o$,  completing the proof.
\end{proof}

\subsection{A pseudo-cycle criterion.}  In practice, one needs a method for verifying  condition \eqref{pseudocyclecondition} for continuous maps
\bear\label{4.3.1}
 f:M\to Z
\eear
with $M$ and $Z$ as in \eqref{3.3}.   To that end,  we impose  various regularity conditions on $\Omega_f$.

\medskip
 
\noindent{\bf Pseudo-cycle criterion.}   Assume that $\Omega_f$ is covered by the images of countably many  maps $\phi_n:U_n\to Z$ where, for each $n$, each $U_n$ is  a  $\sigma$-compact topological manifold of dimension $\le d-2$, and at least one of the following holds:
\begin{enumerate}\setlength{\itemsep}{4pt}
\item[(a)] $\phi_n$ is continuous and locally injective,
\item[(b)]  $\phi_n$ is locally Lipschitz,
\item[(c)]  $\phi_n$ is a $C^1$ map between $C^1$ manifolds.
\end{enumerate}

 \vspace{3mm}
 
  Lemma~\ref{LemmaA2} immediately implies:
 
 \begin{lemma}
\label{Lemma4.3A}
Any map \eqref{4.3.1} that satisfies the pseudo-cycle criterion is a pseudo-cycle.
\end{lemma}

Note that no higher regularity on $f$ itself is needed.  Sometimes Condition (c), with $C^1$ replaced by $C^\infty$,  is  used in the definition of pseudo-cycle in place of \eqref{pseudocyclecondition}  (cf. \cite[\S6.5]{ms2}).

\subsection{Pseudo-cycles and thin compactifications.}   Pseudo-cycles and maps  from thin compactifications both determine  Steenrod homology classes. 
The next  lemma gives conditions under which these classes coincide; the corresponding fact for families is given in Theorem~\ref{theorem.A}. 
  
 \begin{lemma}
 \label{pseudocycleTheorem}
If a map  $f:M\to Z$ as in \eqref{3.3} extends to a continuous map $\ov{f}:\ov{M}\to Z$ from a thin compactification $\ov{M}=M\cup S$ and $\dim  \ov{f}(S)\le d-2$, then $f$ is a pseudo-cycle and
\bear\label{f=ov.f}
 [f]=\ov{f}_*[\ov{M}] \quad \mbox{\rm in}\  \sHH_d(X),
 \eear
 where $X=\ov{f}(\ov{M})$ is the image of $\ov f$.
  \end{lemma}

 \begin{proof} Because $S$ is compact (cf. Definition~\ref{Defn1.1}), the assumptions imply that $\ov{f}(S)$ is a compact subset of the metric space $Z$, while  $\Omega_f\subseteq \ov f(S)$ by Lemma~\ref{Lemma4.1B}.  Lemma~\ref{DimensionTheoryLemma}a) then shows that $\dim \Omega_f\le\dim \ov f (S)\le d-2$,  so $f$ is a pseudo-cycle.

As in \eqref{3.om.s} and \eqref{6.YYFrseq}, the assumption that 
$\dim  \ov f(S)\le d-2$ implies that  the restriction to $U=X\Setminus \ov f(S)$ 
 \best
\begin{tikzcd}
\sHH_d(X)  \ar{r}{\rho_U}[swap]{\cong}&   \sHH_d(U)
\end{tikzcd}
\eest
is an isomorphism. 
Thus it suffices to show that  the two sides of \eqref{f=ov.f} are equal when restricted to $U$.  Using the notation of \eqref{4.MU}, the
restriction of $\ov f$ over $U$ is equal to the restriction $f_U:M_U\ra U$ of $f$ over $U$ (because 
$\ov f^{-1}(U)\subseteq \ov M\Setminus S=M$).  Therefore 
\best
\rho_U \ov f_*[\ov M]= f_{U*} \rho_{M_U}[\ov M]=f_{U*}[M_U],
\eest
which is equal to $\rho_U[f]$ by  \eqref{Lemma3.4eq}.
 \end{proof}

%%%%%%%%%%%%%%%%%%%%%%%%%%%%%%%%%%%%%%%%%%%%%%%%%%%%%%%%%%%%
%%%%%%%%%%%%%%%%%  Section 4  %%%%%%%%%%%%%%%%%%%%%%%%%%%%%%
%%%%%%%%%%%%%%%%%%%%%%%%%%%%%%%%%%%%%%%%%%%%%%%%%%%%%%%%%%%%%%
\setcounter{equation}{0}
\section{Intersection pairings of Pseudo-cycles } 
\label{section4}
\medskip

In an oriented  differentiable manifold, there are two ways to define the intersection between two pseudo-cycles $f$ and $g$ of complementary dimension.   A {\em geometric} intersection $f\cdot g$  is obtained by perturbing the maps to make them (strongly) transverse, and then counting the intersection points with sign (cf.   \cite[\S 6.5]{ms2}).  A {\em homological} intersection $[f]\bullet [g]$ is obtained  by applying the  intersection pairing in Steenrod homology to the pseudo-cycles classes defined by  Lemma~\ref{pseudocycleclass}.  After reviewing these definitions, we give conditions under  which  these two intersection pairings are equal.

\smallskip

 Throughout this section,  all manifolds are $C^1$, oriented,  finite-dimensional and $\sigma$-compact (as in   \cite[\S 6.5]{ms2}).   Each such manifold is separable and  metrizable, and admits a proper embedding into euclidean space (see Appendix~\ref{sectionA}).   We fix one such manifold $N$  of dimension $n$, and consider $C^1$ maps
\bear\label{New4.0}
f:M \to N  \qquad g:P\to N
\eear
 from manifolds $M$ and $P$.  We will call these ``$C^1$ pseudo-cycles'' if they are pseudo-cycles in the sense of McDuff and Salamon \cite[\S 6.5]{ms2}. % Defn 6.5.1
  Thus a $C^1$ map  $f:M\ra N$ between oriented manifolds  is a {\em $C^1$ pseudo-cycle} if  
\begin{enumerate}\setlength{\itemsep}{4pt}
\item[(i)] $\ov {f(M)}$ is compact, and 
\item[(ii)] $\Omega_f$ is covered by the image of a $C^1$-map $f':M'\ra N$ with $\dim M'\le \dim M-2$. 
\end{enumerate}

\subsection{Geometric intersections.}    As in \cite[\S 6.5]{ms2}, we say that  two $C^1$  pseudo-cycles   $f$ and $g$  are {\em strongly transverse} if  $f$ is transverse to  $g$ as $C^1$-maps,  and
 \bear\label{str.tr}
 \Omega_f \cap \ov {g(P)}=\emptyset, \quad \Omega_g \cap \ov {f(M)}=\emptyset.
 \eear 
 If, in addition,   $f$ and $g$ have complementary dimensions ($\dim M+\dim P=\dim N$),  then   
  $$
  Z\, =\, \big\{(x, y)\,|\, f(x)=g(y)\big \}\, = \, (f\ti g)^{-1}(\Delta)
  $$
 is a compact oriented 0-dimensional manifold.  The 
{\em  geometric intersection number} is then defined by
 \bear\label{f.dot.g.MS}
 f\cdot g= \sum_{z\in Z} \sigma(z),
 \eear
 where  $\sigma(z)=\pm 1$ is the sign of the local intersection of $f$ and $g$ at $f(x)=g(y)$.   In fact, they show that  given two complementary dimension $C^1$-pseudo-cycles, then $g$ can be perturbed (by a  $C^1$\!-\,small diffeomorphism of the target) to make it strongly transverse to $f$, and that the resulting  intersection number 
\bear\label{def.geom.int.1}
f\cdot g\in \Z
\eear
is independent of choices, and depends only on the cobordism classes of $f$ and $g$ \cite[Lemma~6.5.5]{ms2}. 
 
\subsection{Homological intersections.}  Appendix~\ref{sectionA}  relates Steenrod to other homology theories and describes features of  Steenrod homology beyond those described in Section~1.   Using coefficients $\Z$ or $\Q$,
 for any closed subset $X$ of the manifold $N$ as in \eqref{New4.0}, there are natural isomorphisms  
\bear\label{5.1}
  \sHH_*(X)  \cong    H^{n-*}(N, N\Setminus X) \cong \HBM_*(X) 
\eear
between Steenrod homology, relative singular cohomology,  and Borel-Moore homology \eqref{A.D1} and \eqref{A.D2}. By these isomorphisms, the intersection theory facts stated in Appendix~\ref{sectionA} for Borel-Moore homology carry over to  Steenrod homology.   Thus, while we work with Steenrod homology, 
all results in this section hold   with Steenrod replaced by Borel-Moore homology.

\smallskip

Under \eqref{5.1}, the cup product in singular cohomology corresponds to a Borel-Moore  intersection pairing \eqref{A.BM.fat.int}, and hence to a Steenrod intersection pairing
\bear\label{5.I2}
 \xymatrix{
\sHH_d(X) \otimes \sHH_{k}(Y)\ \ar[r]^\bullet & \sHH_{d+k-n}(X\cap Y).
}
\eear
This is natural under restriction  as in  \eqref{rho.dot}: for each open subset $U$ of $N$
\bear\label{5.Property1}
\rho_U (a\bullet b)= \rho_U(a )\bullet \rho_U (b).
\eear
In particular, if $X\cap Y$ is compact and we use $\Z$ coefficients, we get an intersection number 
\bear\label{4.epbullet}
a \cdot b = \ep(a\bullet b)\in \Z.  %%% or  \circledbullet  in stix?
\eear
between classes of complementary dimensions (cf. \eqref{int.number}).

\subsection{Relating intersection pairings.} Now consider two $C^1$ pseudo-cycles \eqref{New4.0} with complementary dimensions $d$ and $n-d$. By Lemma~\ref{pseudocycleclass}, these  determine homology classes 
 \best
 [f]\in\sHH_d(X),   \qquad   [g]\in\sHH_{n-d}(Y),  
 \eest 
in the closed subsets  $X=\ov {f(M)}$ and $Y=\ov {g(P)}$ of $N$.  The $C^1$ pseudo-cycle condition (i) above ensures that $X\cap Y$ is compact, so there is a {\em homological intersection number}
$$
[f] \cdot [g] \in\Z,
$$
defined by \eqref{4.epbullet}, which we can compare to the geometric intersection number \eqref{def.geom.int.1}.

 \begin{prop} 
\label{Dot=dot} If  $f$ and $g$ as in \eqref{New4.0} are $C^1$ pseudo-cycles of complementary dimension,   then their geometric and homological intersection numbers are equal:
\bear\label{f.dot.g=f.bullet.g}
 f\cdot g \ =\ [f]  \cdot    [g].     
\eear
\end{prop} 
\begin{proof} 
 First assume that $f$ and $g$ are strongly transverse.  Then by \eqref{str.tr},  we can choose an open neighborhood $U$ of $X\cap Y$ which is disjoint from $\Omega_f\cup \Omega_g$. 
Restricting to $U$ and using \eqref{5.Property1}, we have 
 \best
 [f] \bullet  [g] \ =\ \rho_U\left( [f]\bullet  [g]\right)\ =\  \rho_U[f]\bullet \rho_U[g].
 \eest
  Because   $X \cap U$ is an open subset of $X$ that does not intersect $\Omega_f$, Lemma~\ref{Lemma3.1} and equation \eqref{Lemma3.4eq}  imply that    $f$  restricts to a proper map $f_U: f^{-1}(U)\ra U$ with
 \best
 \rho_U [f]= [f_U], 
 \eest
 and similarly for $g$.   A  direct computation (see Example~\ref{A.Ex3}) shows that, with the notation of \eqref{f.dot.g.MS},
 \best
[f_U] \cdot [g_U] \ =\  \sum_{z\in Z} \sigma (z)\ =\  f\cdot g.  
 \eest
Thus \eqref{f.dot.g=f.bullet.g} holds  in the case that  $f$ and $g$ are strongly transverse.

\smallskip
 
 For the general case, note that both sides of \eqref{f.dot.g=f.bullet.g}  are invariant under  proper cobordism. Indeed,  if $G:W\to N$ is a cobordism between $g$ and $g'$ whose image has  compact closure $B=\ov{G(W)}$, then 
\best
[g]=[g']\quad \mbox{ in } \sHH_k(B)
\eest
by  Lemma~\ref{L.indep.cob}. Therefore 
\best
[f]\bullet [g] = [f]\bullet [g']\quad \mbox{ in } \sHH_0(B\cap X),
\eest 
where $X= \ov {f(M)}$ and where $B\cap X$ is  compact. We conclude that  $[f]\cdot [g]$,  like  \eqref{def.geom.int.1}, is  invariant under proper cobordisms of $g$.

Finally, as in Lemma~6.5.5 of \cite{ms2}, one can perturb $g$ (by a  compactly supported  diffeomorphism of the target  that is isotopic to the identity) to make $g$ strongly transverse to $f$.  This  gives a cobordism $G$, and the proposition follows. 
\end{proof} 

\subsection{Pairing with singular homology.}  Returning to the manifold $N$ in \eqref{New4.0},  there is also a natural transformation
\bear\label{4.phi}
\phi: H_*(N;\Z) \to \sHH_*(N;\Z)
\eear
from singular to Steenrod homology (cf. \eqref{A.trianglediagram}).  Every  class in the image of $\phi$ can be represented by a pseudo-cycle.  The construction, given in  \cite[\S 9.6 and \S 10.2]{ma} and \cite[Remark~6.5.3]{ms2},  can be summarized as follows.
 
\medskip

  Each  $b\in H_{k}(N;\; \Z)$ can be represented   by a singular cycle   $\ov{g}:Q\to N$ from  a $k$-dimensional  finite simplicial complex without boundary. After smoothing across the $(k-1)$-faces, we can assume that $Q=M\cup S$, where 
$M$ is an oriented $k$-dimensional  $C^1$ manifold  and $S$ is the $(k-2)$-skeleton.    Then \eqref{4.phi} is defined by
\bear\label{phi=g}
\phi(b)=\ov{g}_*[Q].
\eear
By further  smoothing,     we can assume
that the   restriction of $\ov{g}$ to each simplex of $Q$ is $C^1$, and  therefore the restriction of $\ov{g}$ to   $M$ is 
a $C^1$-pseudo-cycle 
\bear\label{g'=}
g = \ov{g}\big|_M:M \ra N.
\eear 

\medskip

In this setting,  McDuff and Salamon  associate to each  $C^1$ pseudo-cycle $f:M \ra N$ a homomorphism 
\bear\label{def.phi.f}
\Phi_f: H_*(N;\Z)\ra \Z
\eear
from  singular homology defined by  the geometric intersection pairing \eqref{f.dot.g.MS}:
\bear\label{4.def.phi} 
\Phi_f(b)= f \cdot g.
\eear
They show  that this is  independent of the representative  $\ov{g}$ of $b$ and its smoothing \cite[Lemma~6.5.6]{ms2}. 
In fact,  $\Phi_f$ can be written terms of   the topological  intersection pairing \eqref{4.epbullet} in Steenrod  homology as follows.

\begin{prop} 
For each $C^1$ pseudo-cycle  $f$,  the pairing \eqref{4.def.phi} satisfies 
\bear\label{5.PhiSteenrod}
\Phi_f(b)\ =\ [f] \cdot \phi(b)  \qquad \forall b\in H_*(N;\; \Z).
\eear
 \end{prop}
\begin{proof}
Given $b\in H_k(N;\; \Z)$, choose  a map $\ov{g}:Q\to N$  as in \eqref{phi=g} and a smoothing as in \eqref{g'=}.  Note that $Q$ is the disjoint union of $M$ and the $(d-2)$-skeleton $S$ of $Q$, so $Q$ is a thin compactification of $M$. Moreover,  $g(S)$ is covered by the images of $C^1$-maps from the cells of the $(d-2)$-dimensional skeleton of $Q$. Hence, by Lemmas~\ref{Lemma4.3A} and \ref{pseudocycleTheorem}, the pseudo-cycle class  of $g$ is
\bear\label{g=phi(a).2}
[g] =\ov{g}_*[Q] =  \phi(b).
\eear
Then by \eqref{4.def.phi} and Proposition~\ref{Dot=dot}, 
$\Phi_f(b) =  f \cdot g = [f] \cdot [g]  = [f] \cdot \phi(b)$.
\end{proof}

\medskip

Finally, if $N$ is closed,  \eqref{5.PhiSteenrod} can be translated into a pairing in singular theory:

\begin{cor} 
\label{4.lastcor}
 Let $N$ be a   closed  manifold as in \eqref{New4.0}.   Then for each $C^1$ pseudo-cycle  $f$, there is   a  singular homology class $a_f$ such that $\phi(a_f)=[f]$ and  
$$
\Phi_f (b) \ =\  a_f \cdot b\ =\ \langle b, \;  PD^{-1}a_f \rangle
$$
 for each $b\in H_*(N; \;\Z)$,  where $PD:H^*(N)\to H_*(N)$ is  Poincar\'{e}  duality.
 \end{cor}

\begin{proof} 
For compact manifolds,   \eqref{4.phi} is an isomorphism (cf. \eqref{A.H_*N}),  so  $[f]= \phi(a_f)$ for some unique $a_f\in H_*(N)$. Starting with this class $a_f$, we can repeat the procedure of  \eqref{phi=g} and \eqref{g'=} to obtain an $\ov{h}:P\ra N$, such that, as in \eqref{g=phi(a).2}, 
\best
[h]\;= \;  \ov{h}_*[P] \;= \;  \phi(a_f) =[f],
\eest
where the $C^1$-pseudo-cycle $h$ is the restriction of $\ov{h}$ to the complement of the codimension 2 skeleton of $P$. Then
\best
\Phi_f (b) \;=\;  [f] \cdot \phi(b) \;= \; \Phi_{h} (b) \;=\; h\cdot g
\eest
by \eqref{4.def.phi} and \eqref{5.PhiSteenrod}. But after a further perturbation, $\ov{h}$ and $\ov{g}$ are transversely intersecting singular cycles that represent $a_f$ and $b$.  Hence
 \best
h\cdot g  \; = \; \ov{h}_*[P] \cdot  g_*[Q] \;= \;  a_f\cdot b= \lg b, PD^{-1}a_f\rg,  
\eest
using the fact  that, in singular theory,  intersection numbers depend only on homology and are given by  Poincar\'{e} duality  (we use the same sign conventions as in \eqref{int.ep}).  
\end{proof}

%%%%%%%%%%%%%%%%%
%%%%%%%%%%%%%%%%%
%%%%%%%%%%%%%%%
\setcounter{equation}{0}
\section{Intersections and thin families} 
\label{section5}
\medskip

We now apply the results of Sections~3 and 4 to  Fredholm thin families,  and show that the invariants defined  in Corollary~\ref{1.invtcor} using relative fundamental classes are equal to those defined in terms of  intersections of  pseudo-cycles, as is done, for example,  in \cite{ms2}.  All of the spaces in this section and the next are assumed to be metrizable, and we assume that   $N$ is a  compact manifold without boundary, so there are  natural   isomorphisms (see \eqref{A.H_*N})
\bear\label{5.New1}
\xymatrix{
H_*(N; \Z) \ar[r]_\cong  & \sHH_*(N;\Z) \ar[r]_\cong^\gamma  & \cHH_*(N;\Z).
}
\eear
\medskip

Fix a relatively oriented Fredholm family of index~$d$ (cf. \eqref{2.1}), together with a $C^1$  map $f$ to a  closed  oriented  $C^1$ manifold $N$ of dimension $n$:
$$
\xymatrix{
\M \ar[d]^{\pi} \ar[r]^f & N\\
\P &
}
$$
Also fix  a Fredholm thin compactification $\ov\M$  of $\M$ as in Definition~\ref{1.Def1.1}.
\bear\label{5.MNP}
 \xymatrix{
\Sa \ar[r]^{\phi_\al} \ar[dr]_{\pi_\alpha}& \ov{\M} \ar[d]^{\ov\pi} \ar[r]^{\ov f} &N \\
& \P
}\eear
and let $\P^*\subset \P$ be the set of regular parameters defined in \eqref{2.regularvalues}.   In this section we  consider continuous  maps  $\ov{f}:\ov\M \to N$ as in the above diagram, and use the following terminology.

\begin{defn} \label{def5.1}
For a map $\ov{f}:\ov\M \to N$   as in \eqref{5.MNP}, we say $\ov{f}$   is   $C^m$ if the composition  
$\ov f\circ \phi_\alpha:{\cal S}_\alpha\to N$ is $C^m$ for each $\alpha$ (including ${\cal S}_0=\M$). 
\end{defn}
In this situation,  Theorem~\ref{theorem1.2} implies that each fiber $\ov\M_p$ of $\ov\pi$ has a relative fundamental class $[\ov\M_p]^\rfc \in \cHH_d(\ov\M_p)$.  As in \eqref{1.invts},  this  class pushes forward and pairs with  \Cech cohomology, defining a homomorphism
 $$
I_p:  \cHH^*(N;\Z)\to \Z 
 $$
by
\bear\label{5.invts}
I_p(\beta)\ =\ \langle \ov{f}_*[\ov\M_p]^\rfc , \;  \beta \rangle  \ =\ \langle [\ov\M_p]^\rfc, \;  \ov{f}^*\beta \rangle.
\eear
By Corollary~\ref{1.invtcor}, $I_p$ is independent of $p$ on each path-connected component of $\P$.

The number \eqref{5.invts} is often written as
$$
 \int_{[\ov\M_p]^\rfc} \ov{f}^*\beta.
$$
It is commonly interpreted either \setlength{\itemsep}{4pt}
\begin{enumerate}
\item[(1)]  in terms of intersections with pseudo-cycles  in $N$  that represent homology classes Poincar\'{e} dual to $\beta$, or 
\item[(2)] in terms of ``cutdown''  spaces.  
\end{enumerate}
The next theorem shows how the relative fundamental class approach is compatible with interpretation (1).   Section~\ref{section6}  addresses  compatibility with interpretation (2).
 
\bigskip

Given maps as in  diagram \eqref{5.MNP} and $p\in\P^*$, let  $\ov{f}_p:\ov\M_p\to N$ be the  restriction of $\ov{f}$ to the fiber  of $\ov\pi$ over $p$, and its further restriction $f_p:\M_p\to N$.
\medskip

\begin{theorem}\label{theorem.A}  Suppose that the map $\ov{f}$   in \eqref{5.MNP} is  $C^1$ in the sense of Definition~\ref{def5.1}.  Then  for  each $p\in \P^*$, 
\begin{enumerate}[(a)]\setlength{\itemsep}{4pt}
\item $f_p$ is a $C^1$ pseudo-cycle and its pseudo-cycle class \eqref{4.psClass} satisfies
\bear\label{f.p=f.p*}
\gamma [f_p]= \ov{f}_*[\ov\M_p]^\rfc
\eear
 in  $\cHH_d(N;\Z)$ (cf. \eqref{5.New1}).
\item If $N$ is compact,  the topological invariant \eqref{5.invts} is related to the geometric intersection pairing  \eqref{def.phi.f} by 
\bear\label{5.IPhi}
I_p(\beta) \; =  \;   (-1)^{d(n-d)}\,  \Phi_{f_p}(b),
\eear
where $b=D\beta \in H_{n-d}(N)$ is  the Poincar\'{e} dual of $\beta$  under \eqref{A.PD}.
\end{enumerate}
\end{theorem} 

\begin{proof}
(a) Fix  a regular value  $p\in\P^*$.  By Lemma~\ref{lemma1.2}, $\ov\M_p$ is a thin compactification of the oriented $d$-dimensional manifold $\M_p$ with  singular locus
${\cal S}_p=\ov\M_p\Setminus \M_p$.
It therefore  has a Steenrod fundamental class $[\ov\M_p]$ and a corresponding \Cech class  $\gamma[\ov\M_p]$.    The naturality of $\gamma$  and the normalization axiom \eqref{1.A2} then imply that
\bear\label{5.fRFC} 
\ov{f}_*[\ov\M_p]^\rfc \ =\  \gamma \left((\ov{f}_{p})_*[\ov\M_p]  \right)   \quad \mbox{ in } \cHH_d(N;\Z).
\eear

For each $\alpha$, as in the proof of Lemma~\ref{lemma1.2},  the fiber $\pi_\al^{-1}(p)=(\Sa)_p$ is a 
 $\sigma$-compact $C^1$ manifold of dimension at most $d-2$.  Furthermore, the restriction of  $\ov f \circ \phi_\al$ to $(\Sa)_p$  is $C^1$ and 
 $$
 \ov f_p({\cal S}_p)\, \subseteq\, \bigcup_\alpha \ (\ov f_p\circ \phi_\al)(\Sa)_p.
 $$
 Lemma~\ref{LemmaA2} then shows that the compact set $\ov f_p({\cal S}_p)$ has dimension at most $d-2$.
Therefore, by Lemma~\ref{pseudocycleTheorem},  $f_p:\M_p \ra N$ is a $C^1$ pseudo-cycle, and
\best
 [f_p]\ =\ (\ov{f}_p)_*[\ov\M_p] 
\eest
in $\sHH_d(N)$.  Together with  \eqref{5.fRFC}, this yields  \eqref{f.p=f.p*}.

\medskip

(b) Returning to \eqref{5.invts} and using \eqref{f.p=f.p*},  and then using \eqref{A.twopairings} to switch from the \Cech to the Steenrod Kronecker pairing,  we have
$$
I_p(\beta)\ =\ \langle \gamma [ f_p], \; \beta \rangle \ =\   \langle  [ f_p], \; \beta \rangle.
$$
This last expression can be written as an intersection pairing as in \eqref{int.ep}.  Hence, using the definition of $b$ and \eqref{5.PhiSteenrod}, we have
$$
I_p(\beta)\ =\   D\beta  \cdot  [f_p] \ =\   (-1)^{d(n-d)}\, [ f_p] \cdot b  \ =\  (-1)^{d(n-d)}\,  \Phi_f(b).
$$
 \end{proof}

%%%%%%%%%%%%%%%%%
%%%%%%%%%%%%%%%%%
%%%%%%%%%%%%%%%
\setcounter{equation}{0}
\section{RFCs for Cutdown Families} 
\label{section6}
\medskip

We now turn to the second interpretation of  the intersection number \eqref{5.invts}.  The basic idea is that for a generic submanifold $V$ of $N$, the inverse image  of the map $\ov{f}:\ov\M\to N$ should be a ``cutdown'' family $\ov \V\to\P$ that has all of the properties of the original relatively thin family $\ov\M\to \P$;  in particular it should have a relative fundamental class.  We will show that this is true provided $f$ is ``fully transverse to $V$''.  If, in addition, the cutdown family has index~0,    the invariant \eqref{5.invts}  has the expected  geometric interpretation:  it is the signed number of elements of a generic fiber of this index~0 family,  i.e. those in a generic fiber of the original family whose images under $f$  lie in $V$.

This section gives the details.  We begin with a standard  formula \eqref{6.[XV]} that relates the fundamental classes of a manifold and a submanifold.  We then extend it first to (metrizable) thin compactifications of finite-dimensional manifolds, and then to Fredholm thin families.

\subsection{Orientation classes.}  Let $X$ be a closed subset of a locally compact  space $Z$.  Then there is a cap product   
\bear\label{5.Qcap.BM}
\xymatrix@=5mm{
  \HBM_d(Z)\otimes H^k(Z, Z\Setminus X)   \ar[rr]^{\qquad\quad \cap}     &&     \HBM_{d-k}(X)}
\eear
in Borel-Moore homology  with coefficients in $\Z$ or $\Q$  (see \eqref{A3.1}) with the naturality properties \eqref{A.cupnaturality1}--\eqref{A.cupnaturality3}.

Now suppose that $N$  is an oriented $C^1$ manifold of dimension $n$ and 
\bear\label{6.VinN}
V\hookrightarrow N
\eear
 is  a properly embedded oriented  submanifold of codimension $k$.  The orientations determine Borel-Moore fundamental classes $[N]\in \HBM_n(N)$ and $[V]$ respectively (see \eqref{fund.BM.sheaf}), and the Thom class of the normal bundle defines an  {\em orientation class}
\bear\label{5.transverseclass1}
u=u_{V, N}\in  H^{k}(N, N\Setminus V) 
\eear
 in    singular cohomology  (see \eqref{A.transverseclass}). As in \eqref{A.[XV]}, these are related by
\bear\label{6.[XV]}
[V]\ =\ [N] \cap u_{V, N} \quad \mbox { in } \HBM_{n-k}(V).
\eear

\subsection{Thin compactifications.}  
  Formula~\eqref{6.[XV]} extends  to metrizable thin compactifications as follows.     Fix  $V \subseteq N$  and $u$ as above.  Suppose that $\ov{V}\subseteq \ov{N}$ are
 thin  compactifications   of   $V$ and $N$ such that  
\begin{enumerate}[(i)]\setlength{\itemsep}{4pt}
\item  $V=\ov{V}\cap N$, and
\item the orientation class \eqref{5.transverseclass1}  is the restriction to $N$ of some class
 $\ov{u}\in H^*(\ov N, \ov N\Setminus \ov V)$, i.e. $ u=j_N^*\ov{u}$ where $j_N:N\to\ov{N}$ is the inclusion.
\end{enumerate}

 \begin{lemma}\label{L.smfld.thin} 
 If (i) and (ii) above hold and $\ov N$ is metrizable,  then 
 \bear\label{V=tau.cap.M}
[\ov V]= [\ov N]\cap  \ov u \quad \mbox { in }  \HBM_{*}(\ov V; \;\Z)=\sHH_{*}(\ov V;\; \Z).
 \eear
 \end{lemma}

\begin{proof}   
Because $\ov{N}$ and  $\ov{V}$ are    compact and metrizable, their Steenrod and Borel-Moore homologies are isomorphic  \cite[\S5]{BM} (see also \S A.3 below).
For the proof, we work in Borel-Moore homology, using  \eqref{6.[XV]} and noting that  equations \eqref{1.rhoU}-\eqref{1.fc} hold in Borel-Moore homology (cf. \S A.2 and Lemma~\ref{LemmaB1}).  In particular, the restriction 
$\rho_V:   \HBM_{n-k}(\ov V)\to \HBM_{n-k}(V)$ is an isomorphism,  $\rho_V[\ov V]=[V]$ and $\rho_N[\ov N]=[N]$. Therefore it suffices to show that  the two sides of  \eqref{V=tau.cap.M} are equal after applying $\rho_V$.   But 
assumption (i), and the naturality \eqref{A.cupnaturality3}  of the cap product with respect to the restriction to $N$, and assumption~(ii) imply that
\best
\rho_V( [\ov N] \cap\ov u )\,=\, \rho_{N\cap \ov V}( [\ov N]\cap \ov u )\,=\, \rho_N [\ov{N}] \cap j^*_N(\ov u)  =   [ N]\cap u 
\eest
in $\HBM_*(V)$, and this is equal to  $[V] =\rho_V [\ov V]$ by \eqref{6.[XV]}.   
\end{proof}

\subsection{Cutdown families.} Fix a map  $\ov{f}:\ov\M \to N$ as in \eqref{5.MNP} and a codimension $k$ submanifold $V$ of $N$ as in \eqref{6.VinN}.  This data determines a  ``cutdown'' family 
\bear\label{5.cutdown}
\xymatrix{
\V \ar@{^{(}->}[r]  \ar[rd] &  \oV  \ar[d]^{\pi} & \hspace{-9mm} = \ov f^{\; -1}\!(V) \\
& \P &
}
\eear
obtained by restricting $\ov\pi$ to $ \ov f^{\; -1}\!(V)$, setting $\V=\oV\cap\M$, and giving $\V$  the  induced relative orientation.  Recall that the index $d$ family $\pi:\M \ra\P$ satisfies  \eqref{2.lbound}, and we assume that $f$ is $C^m$, where
\best
m > \max(d-k,0).
\eest

\begin{defn} \label{def6.1}
We say that  $\ov{f}:\ov\M \to N$  is  {\em fully  transverse} to a submanifold $V$ if  $\ov f\circ \phi_\alpha$ is transverse to $V$ for all $\alpha$.
\end{defn}

\begin{lemma}
\label{L.6.cut.down} With $N, V$ as above, if a map $\ov{f}$ as in \eqref{5.MNP}   is  $C^m$ and fully transverse to $V$, then the cutdown family \eqref{5.cutdown} is a Fredholm thin compactification 
of index $d-k$. 
\end{lemma}
\begin{proof} By assumption, each composition $f_\al = \ov{f}\circ\phi_\alpha$   in \eqref{5.MNP}  is a $C^m$ map  transverse to $V$.  Therefore $\V_\alpha=f^{-1}_\al(V)$ is a closed codimension $k$ submanifold of ${\mathcal S_\al}$, and $\ov\pi$ restricts to a Fredholm map $\V_\alpha\to \P$ whose index is $(\ind\; \pi_\al)-k$.
 In particular,  the top stratum $\V=\oV\cap \M \to \P$ is a Fredholm family of index~$d-k$, and all other $\V_\alpha$ have index at most $d-k-2$. 
This implies that   the restriction of $\ov\pi$ to $\oV$ is a Fredholm thin compactification of $\V$ with index $d-k$ (cf. Definition~\ref{1.Def1.1}). 
\end{proof}

\subsection{The cutdown RFC}  In the remainder of this section, we will extend formula \eqref{V=tau.cap.M}  to cutdown families.  For this, we pass to rational coefficients and use the  natural identifications 
$$
\sHH_*(X;\Q)=\HBM_*(X;\Q)=\cHH_*(X;\Q)
$$
for compact metric spaces (see \ref{A.6}).
Under this identification \eqref{5.Qcap.BM}  becomes a cap product in rational \Cech homology.

\smallskip

For each parameter $p\in \P$, let $\ov{{\mathscr V}_p}$ and $\ov\M_p$ denote the fibers of the two families \eqref{5.cutdown} and \eqref{5.MNP} respectively, and $\ov f_p:\ov\M_p\ra N$ the restriction of $\ov f$ to $\ov\M_p$.  Then the pullback of \eqref{5.transverseclass1} is a class
\bear \label{pull.back.fp}
\ov f_p^*u\in H^k (\ov \M_p, \ov \M_p \Setminus \ov \V_p),
\eear
and hence the cap product  \eqref{5.Qcap.BM} induces a map 
% (sheaf theoretic) cap product on Borel-Moore homology
\bear\label{capf*u.3}
 \ \cap\  {\ov f}_p^* u\ :  \cHH_d (\ov \M_p;\; \Q) \longrightarrow \cHH_{d-k}( \ov \V_p;\; \Q). 
\eear

 \begin{prop}
\label{P.6.cut.down}
Let  $V\hookrightarrow N$ as in \eqref{6.VinN}  with orientation  class \eqref{5.transverseclass1}. Assume a map $\ov{f}$ as in \eqref{5.MNP}   is  $C^1$ and fully transverse to $V$. Then the cutdown family \eqref{5.cutdown} has a relative fundamental class 
\bear\label{6.rfcVf}
[\ov{{\mathscr V}_p}]^\rfc \in  \cHH_*(\ov{{\mathscr V}_p}; \Z)
\eear
which is related to the relative fundamental class of 
$\ov \M$ by 
\bear\label{4.cut-down}
[\ov{{\mathscr V}_p}]^\rfc \ =\  [\ov \M_p]^\rfc \cap {\ov f}_p^* u 
\eear
in $\cHH_*(\ov{{\mathscr V}_p}; \Q)=\HBM_*(\ov{{\mathscr V}_p}; \Q)$  for all $p\in \P$.  
\end{prop}
\begin{proof}
Lemma~\ref{L.6.cut.down} and Theorem~\ref{theorem1.2} immediately imply the  existence of the relative fundamental class \eqref{6.rfcVf}.    To establish \eqref{4.cut-down},  
 apply Lemma~\ref{lemma1.2}  to both $\ov\M$ and $\oV$  and intersect the resulting sets  \eqref{2.regularvalues} of regular values.  This yields a second category subset $\P^{**}$ of $\P$ such that, for each $p\in\P^{**}$, 
\begin{enumerate}[(i)]  
\item $\ov \M_p $ is  a metrizable thin compactification of $\M_p$.\\[-3mm]
\item $\ov{{\mathscr V}_p}$ is a metrizable thin compactification of $\V_p$. \\[-3mm]
 \item $\V_p= f_p^{-1}(V)$ is a  oriented embedded submanifold of $\M_p$ whose orientation class, as in   \eqref{A.nat.Thom}, is 
 $f_p^* u$, where $u=u_{N,V}$.
 \end{enumerate} 
 Now fix $p\in\P^{**}$\!. The naturality of the restriction map shows that $f_p^* u$ is the restriction to $\M_p$ 
 of the class \eqref{pull.back.fp}. Therefore Lemma~\ref{L.smfld.thin} applies, giving 
 \bear\label{M.cap.tau=V}
[\ov \V_p] \, =\,[\ov \M_p]  \cap  \ov f_p^*u \quad \mbox{in } \HBM_* (\ov \V_p, \Q)= \cHH_* (\ov \V_p, \Q),  
\eear
for each $p\in\P^{**}$.   To complete the proof, observe that each side of \eqref{M.cap.tau=V} is  a relative homology functor associated to $\oV \to \P$:
\begin{itemize}
\item  By Theorem~\ref{theorem1.2},  $\mu= [\ov\V]^\rfc$ is a relative homology functor  with
$$
\mu(p)= [\ov\V_p]^\rfc \in  \cHH_* (\ov \V_p, \Q) 
$$

\item    The calculation \eqref{1.capcalc} used to show that  \eqref{1.20} is relative homology functor, now using the cap product \eqref{capf*u.3} with its naturality property \eqref{A.cupnaturality1}, shows that   $\tilde\mu= [\ov\M]^\rfc \cap \ov{f}^*u$   is a relative homology functor  with
$$
\tilde\mu(p)= [\ov\M_p]^\rfc \cap \ov{f}_p^*u\in  \cHH_*(\ov \V_p, \Q) 
$$
\end{itemize}
But  \eqref{M.cap.tau=V} shows that
\best
\mu(p)=\tilde{\mu}(p)\quad \mbox { for all } p\in  \P^{**}. 
\eest
By Proposition~\ref{2.ExtensionLemma}(a), we conclude that $\mu$ and $\tilde\mu$ are equal, giving \eqref{4.cut-down}.
 \end{proof} 
 \medskip

\begin{cor} Under the assumptions of Proposition~\ref{P.6.cut.down}
\bear\label{Cor6.5eq}
\int_{[\ov \V_p]^\rfc } \iota^*\alpha \ =\  \int_{[\ov \M_p]^\rfc }  {\ov f}_p^* u  \cup \alpha  \hspace{12mm}  \forall \alpha\in H^*(\ov\M_p;\Q)
\eear
where $\iota$ is the inclusion $\ov\V_p\hookrightarrow\ov\M_p$. 
\end{cor}
\begin{proof} 
Applying the naturality formula   \eqref{A.cupnaturality2} and  \eqref{4.cut-down} gives
\best
[\ov \V_p]^\rfc\cap\iota^*\alpha  \ =\  ( [\ov \M_p]^\rfc\cap \ov f^*_p u)\cap \iota^*\alpha\ =\ [\ov \M_p]^\rfc\cap (\ov f^*_p u\cup \alpha ), 
\eest
which is equivalent to \eqref{Cor6.5eq} (cf.  \eqref{int.ep}).
\end{proof}

%%%%%%%%%%%%%%%%%%%%%%%%%%%%%%%%%%%%%%%%%%%%%%%%%%%%%%%%%%%%
%%%%%%%%%%%%%%%%%  Section 7  %%%%%%%%%%%%%%%%%%%%%%%%%%%%%%
%%%%%%%%%%%%%%%%%%%%%%%%%%%%%%%%%%%%%%%%%%%%%%%%%%%%%%%%%%%%%%
\setcounter{equation}{0}
\section{VFCs defined by implicit atlases} 
\label{section7}
\medskip

In \cite{pardon}, John Pardon defined a notion of an ``implicit atlas'' on a compact  Hausdorff  space, and used it to define a  virtual  fundamental class.
After a very brief review of Pardon's setup, we establish a basic fact, Lemma~\ref{Pardon.L} below, about his fundamental class.   This lemma,  and its rather technical proof,
are entirely within Pardon's setup, and can be regarded as a small  addendum to his paper \cite{pardon}.   Lemma~\ref{Pardon.L}  is used in Section~\ref{section8} to relate Pardon's virtual fundamental class to the relative fundamental class. 
 
 \medskip

We start with a brief review of Pardon's implicit atlas package \cite[Definition~3.1.1]{pardon}. Let $X$ be  a compact Hausdorff space.   Roughly speaking, an  {\em implicit atlas} $\A$  on $X$ with index set $A$ organizes a collection of local charts indexed by finite subsets $I=\{\alpha_1,\dots \alpha_k\}$  of $A$ (including $I=\emptyset$). Each chart consists of 
\begin{enumerate}\setlength\itemsep{4pt}
\item[(i)] a ``thickening'' space $X_I$ containing an open subset $X_I^{reg}\subseteq X_I$ that is a manifold, 
\item[(ii)] an ``obstruction space"  $E_I=\oplus_{\al\in I} E_\al$ which is a vector space of dimension $\dim E_I=\dim X_I^{reg}-d$ , and 
\item[(iii)]  ``Kuranishi maps" $s_\al:X_I\to E_\alpha$ for each $\al \in I$. 
\end{enumerate} 
One can also include finite groups $\Gamma_\alpha$ acting linearly on $E_\alpha$ for each $\al\in A$. There is additional data needed to ensure compatibility, e.g. ``footprint'' maps $\psi_{IJ}$ defined for each $I\subseteq J \subseteq A$. These are required to satisfy a host of  compatibility conditions and transversality axioms \cite[Definition~3.1.1]{pardon}.    In particular, the existence of an implicit atlas implies that $X$ is locally metrizable, and hence metrizable since $X$ is compact.

\medskip
Assuming that the implicit atlas is locally orientable in the sense of \cite[Definition~4.1.2]{pardon}, Pardon associates a virtual cochain complex  
\cite[\S4]{pardon}, and uses its homology to define a virtual fundamental class 
\bear\label{P.VFC}
[X]_A^{\mathrm {vir}} \in  \cHH^d(X; \mathfrak o_{X\; \mathrm{rel} \; \bd})^\vee
\eear
in the dual of rational \Cech cohomology  with coefficients in an orientation sheaf.  Specifically, $[X]_A^{\mathrm {vir}}$ is defined in \S 5.1    of  \cite{pardon}  as the composition 
$$
\cHH^d(X; \mathfrak o_{X\; \mathrm{rel} \; \bd}) =H^d_{\mathrm {vir}}(X\; \mathrm{rel} \; \bd;  A) \xra{s_*} 
H_{0}(E,A )\xra{[E_A]\mapsto 1}\Q.
$$

   Pardon defines \eqref{P.VFC} in the general context of an implicit atlas with boundary. 
In the special case of implicit atlases, which  suffices for our purposes,   
 the sheaf $\mathfrak{o}_{X\,\mathrm{rel}\partial}$ could   be written as $\mathfrak{o}_{X}$.
 Nevertheless, we will retain Pardon's    ``$\mathrm{rel}\,\partial$''  notation to facilitate comparison with Sections~4 and 5 of  \cite{Pd}.

\medskip

In general, a space $X$ with an  implicit atlas $A$ contains a distinguished open  subset, the regular locus $X^{reg}_\emptyset$ of $X$,  which is a  $d$-dimensional manifold (
 without boundary, but not necessarily compact).    Along the regular locus,  Pardon's virtual orientation sheaf $\mathfrak o_{X\; \mathrm{rel} \; \bd}$ is canonically identified with the orientation sheaf of $X^{reg}_\emptyset$ as a manifold \cite[Definition 4.1.3]{pardon}.  

\medskip

The regular locus $X^{reg}_\emptyset $ is an open subset of the compact set $X$, so the inclusion
\bear\label{incl.j}
j:X^{reg}_\emptyset \hookrightarrow X
\eear
induces a map in compactly supported \Cech cohomology, and hence a diagram
\bear\label{j!}
\xymatrix@C=4mm@R=4mm{
   \cHH^d_c(X;j_!j^*\mathfrak o_{X\; \mathrm{rel} \; \bd})\ar[r]  &
\cHH^d_c(X; \mathfrak o_{X\; \mathrm{rel} \; \bd}) \ar[d]^= \\
\cHH_c^d(X^{reg}_\emptyset; \mathfrak{o}_{X\; \mathrm{rel} \; \bd}) 
\ar[u]_{\cong}^{j_!}  &  \cHH^d(X; \mathfrak o_{X\; \mathrm{rel} \; \bd}) 
}
\eear
 where the horizontal map is induced by  the map of sheaves $j_!j^*\F\ra\F$, and the vertical maps are as in Lemmas A.4.7 and A.4.5 of \cite{pardon}.  
  The dual of the composition is a map 
 \bear\label{ourinclusionj}
\rho_{reg}: \cHH^d(X; \mathfrak o_{X\; \mathrm{rel} \; \bd})^\vee \ \overset{\ \ \ }\longrightarrow\,   \cHH^d_c(X^{reg}_\emptyset; \mathfrak o_{X\; \mathrm{rel} \; \bd})^\vee.
\eear

\medskip

Now assume for simplicity that the regular locus  $X^{reg}_\emptyset$ is oriented.  Then  it carries a fundamental class in rational  Steenrod homology 
$$  
[X^{reg}_\emptyset]\in  \sHH_d(X^{reg}_\emptyset,\Q).
$$ 
The natural isomorphism $\tau$ defined by   \eqref{A.tau} takes this to the fundamental class 
\bear\label{VFC.sc}
[X^{reg}_\emptyset] \in  \cHH^d_c(X^{reg}_\emptyset; \Q)^\vee
\eear
in the dual of  compactly supported rational \Cech cohomology.  
 The    orientation  also determines an   isomorphism
 \bear\label{ourinclusionjNEW}
 \cHH^d_c(X^{reg}_\emptyset; \mathfrak o_{X\; \mathrm{rel} \; \bd})^\vee \overset{\ \ \cong\ \ }\longrightarrow  \cHH^d_c(X^{reg}_\emptyset; \Q)^\vee.
\eear

\medskip

 In the special case that the regular locus is all of $X$,  $X=X^{reg}_\emptyset$ is 
 a topological  manifold, and Pardon shows that the class \eqref{P.VFC} determined by  an orientable implicit atlas is equal to the 
  usual fundamental class \eqref{VFC.sc} \cite[Lemma~5.2.6]{pardon}.  For our purposes, we need the following more general fact, whose  proof was communicated to us by Pardon.  It asserts that, for any orientable implicit atlas,  the restriction of  \eqref{P.VFC} to 
the regular locus $X^{reg}_\emptyset$  is the fundamental class \eqref{VFC.sc}.

\begin{lemma}[Pardon]
\label{Pardon.L} 
 Let  $X$ be a compact Hausdorff space with a $d$-dimensional locally orientable implicit atlas $A$.  Assume that the regular locus
 $X^{reg}_\emptyset$ is oriented.   Then  
\bear\label{4.jinduced}
\rho_{reg}[X]_A^{\mathrm {vir}}\ =\ [X^{reg}_\emptyset]
\eear
 where the righthand side corresponds to  the fundamental class \eqref{VFC.sc}  under   \eqref{ourinclusionjNEW}.  
\end{lemma}
\begin{proof} 
The proof is an exercise understanding the naturality of the  map \eqref{ourinclusionj} in  Pardon's VFC package.   Before starting,  note that  metrizable spaces are  paracompact, and that we can assume that  the atlas $A$ is finite (cf. Section~5.1 of \cite{pardon}).

In the context of Section 4.3 of \cite{pardon}, we have sets $V_I$ defined by 
 $V_I= \psi_{\emptyset I}\big(s_I^{-1}(0)\cap X_I^{reg}\big)$  associated to each subset $I$ of $A$.    In the special case $I=\emptyset$,  axioms (iv) and (vii) in  the definition of the implicit atlas \cite[Definition 3.1.2]{pardon} imply that
  $s_\emptyset=0$ and $\psi_{\emptyset,\emptyset}=id$, and hence  $V_\emptyset= X_\emptyset^{reg}$.  Taking  $K=X$ and $I=J=\emptyset$, we have
\bear\label{H.vir=H.c}
H^d_{\mathrm{vir}}(X\; \mathrm{rel}\; \bd, A)_{\emptyset\emptyset} \ =\ 
\cHH^d_c( X_\emptyset^{reg}; \mathfrak o_{X\; \mathrm{rel} \; \bd} )
\eear
by  \cite[(4.3.9)]{pardon}.   Furthermore,   the map $j:V_I\cap V_J\to X$ used in Section~4.3 of \cite{pardon} is equal to \eqref{incl.j} in the case $I=J=\emptyset$, and hence 
\best\label{H.vir=H.j}
H^d_{\mathrm{vir}}(X\; \mathrm{rel}\; \bd, A)_{\emptyset\emptyset}= \cHH^d(X;j_!j^*\mathfrak o_{X\; \mathrm{rel} \; \bd})
\eest
by  \cite[(4.3.11)]{pardon}.
We then have a commutative diagram
\begin{align}\label{X.P.diag}
\begin{gathered}
\xymatrix@C=4mm@R=4mm{
  \cHH^d(X; \mathfrak o_{X\; \mathrm{rel} \; \bd}) \ar@{=}[r]& 
H^d_{\mathrm {vir}}(X \mathrm{rel}\; \bd;  A)\ar[rr]^{\quad s_*} &
&H_{0}( E;A)\ar[rr]^{\quad [E_A] \mapsto 1}&&\Q   
\\
\cHH_c^d(X^{reg}_\emptyset; \mathfrak{o}_{X\; \mathrm{rel} \; \bd})  \ar[u] \ar@{=}[r]& 
H^d_{\mathrm {vir}}(X \mathrm{rel}\; \bd;  A)_{\emptyset,\emptyset} \ar[u]\ar@{-->}[rru]_(.6){(s_{\emptyset\emptyset})_*}     &
&  && 
}
\end{gathered}
\end{align}
where the   first vertical arrow  is from \eqref{j!}, and the second vertical arrow  is induced by the maps (4.2.8) and (4.2.10) in  \cite{pardon}.    The square in \eqref{X.P.diag} is commutative by the constructions in Section~4.3 of \cite{pardon} (the top  isomorphism is the homotopy colimit of the isomorphisms  (4.3.11), which include the case $I=J=\emptyset$). Diagram \eqref{X.P.diag} shows that  $ \rho_{reg}[X]^{vir}_A$ is  the composition
\bear\label{7.comp.xx}
\begin{tikzcd}[column sep=large]
\cHH_c^d(X^{reg}_\emptyset;  \mathfrak{o}_{X\; \mathrm{rel} \; \bd}) =H^d_{\mathrm {vir}}(X \mathrm{rel} \;\bd;  A)_{\emptyset,\emptyset}   \arrow[r, dashrightarrow, "{(s_{\emptyset \emptyset})_*} "]
&H_0 (E, A)\arrow[r,"{[E_A]\mapsto 1}"] & \Q.  
\end{tikzcd}
\eear

The proof is completed by showing that the composition \eqref{7.comp.xx} is the fundamental class  \eqref{VFC.sc}. We do this by looking at the chain level, using (4.2.3), (4.2.4), (4.2.12) in \cite{pardon} as follows.   In the case $I=J=\emptyset$, Definition~4.2.1 of \cite{pardon} reduces to  
\bear\label{X.00A}
X_{\emptyset\emptyset A}= E_A \ti \{0\}\ti X_\emptyset^{reg},
\eear  
which we identify with  $E_A \ti X_\emptyset^{reg}$. This is an  orientable $d+\dim E_A$  dimensional manifold.    The space $X_{I,J,A}$ in Definition~4.2.1 comes with a $\Gamma_J$ action, which is implicitly  extended to  an action of  $\Gamma_A$ (by the canonical map $\Gamma_A\to\Gamma_J$).  In particular, $\Gamma_A$ acts trivially on  the factor $X_\emptyset^{reg}$ of \eqref{X.00A}.    Similarly, for each $K \subseteq X$, Definition~4.2.1 of \cite{pardon} gives
% (with an action of the trivial group $\Gamma_\emptyset=\{1\}$
\best
X_{\emptyset\emptyset A}^K = \{0\} \ti (K\cap X_\emptyset^{reg}),
\eest 
which we similarly identify with $K\cap X_\emptyset^{reg}$. 
In particular, taking $K=X$ gives the identification
\best
C_{\dim E_A}(X_{\emptyset\emptyset A}, X_{\emptyset\emptyset A}\Setminus X_{\emptyset}^{reg},\mathfrak{o}_{E_A}^\vee)^{\Gamma_A}= C_{\dim E_A}( E_A \ti X_\emptyset^{reg}, (E_A \Setminus 0)\ti X_\emptyset^{reg}, \mathfrak{o}_{E_A}^\vee)^{\Gamma_A}.
\eest

\smallskip

Next, as in   \cite[(4.2.7), (4.2.12)]{pardon}, the map    $s_{\emptyset \emptyset *}$ is  defined as the composition around the square
\best
\xymatrix@C=4mm@R=3mm{
C^d_{\mathrm {vir}}(X \mathrm{rel} \bd;  A)_{\emptyset,\emptyset} \ar[r]^{s_{\emptyset \emptyset*}}&
C_0 (E, A). \\
C_{\dim E_A}(X_{\emptyset\emptyset A}, X_{\emptyset\emptyset A}\Setminus  X_\emptyset^{reg}, \mathfrak{o}_{E_A}^\vee)^{\Gamma_A} \;
\ar@{=}[u]\ar[r]&\;
  C_{\dim E_A}(E_A, E_A \Setminus 0; \mathfrak{o}_{E_A}^\vee)^{\Gamma_A}.
  \ar@{=}[u]
  \\
}
\eest
Here, the vertical arrows are the identifications (4.2.4) and  (4.2.3) in \cite{pardon}, respectively, and the bottom map is induced by the  projection $X_{\emptyset\emptyset A} \to E_A$ onto the first factor of \eqref{X.00A}.    Passing to homology 
and using 
\eqref{H.vir=H.c} gives a commutative diagram 
\best
\xymatrix@C=4mm@R=3mm{
\cHH_c^d(X^{reg}_\emptyset; \mathfrak{o}_{X\; \mathrm{rel} \; \bd})\ar@{=}[r] \ar@{-->}[dr] &\cHH^d_{\mathrm {vir}}(X \mathrm{rel} \bd;  A)_{\emptyset,\emptyset} \ar[r]^{s_{\emptyset \emptyset*}}&
\cHH_0 (E, A)
\ar[r]^{[E_A]\mapsto1}& \Q. \\
&\cHH_{\dim E_A}( X_{\emptyset\emptyset A}, X_{\emptyset\emptyset A}\Setminus X_\emptyset^{reg}; \mathfrak{o}_{E_A}^\vee)
\ar@{=}[u]\ar[r]&
  \cHH_{\dim E_A}(E_A, E_A \Setminus 0; \mathfrak{o}_{E_A}^\vee)
  \ar@{=}[u] \ar@{-->}[ur]
}
\eest
whose top row is \eqref{7.comp.xx}. The first dashed arrow is Poincar\'{e}-Lefschetz duality in the form \cite[(A.6.6)]{pardon} (Pardon's equations (4.3.3)-(4.3.11), used above to obtain \eqref{H.vir=H.c}, reduce to this for $I=J=\emptyset$). Moreover, by \eqref{X.00A} and Kunneth decomposition, the lower left group is $ H_{\dim E_A} (E_A, E_A\Setminus 0) \otimes H_0(X^{reg}_\emptyset)$, and the homology map induces by $X_{\emptyset\emptyset A} \to E_A$ is ${\mathrm id} \otimes \ep$ where $\ep: H_0(X^{reg}_\emptyset)\ra \Q$ is the augmentation. 
It follows that the composition 
\best
\cHH_c^d(X^{reg}_\emptyset;\mathfrak{o}_{X\; \mathrm{rel} \; \bd}) \ra \Q
\eest
along the bottom of the diagram is equal to the fundamental class \eqref{VFC.sc} after trivializing the orientation sheaf  over $X^{reg}$.
\end{proof}

%%%%%%%%%%%%%%%%%%%%%%%%%%%%%%%%%%%%%%%%%%%%%%%%%%%%%%%%%%%%
%%%%%%%%%%%%%%%%%  Section8  %%%%%%%%%%%%%%%%%%%%%%%%%%%%%%
%%%%%%%%%%%%%%%%%%%%%%%%%%%%%%%%%%%%%%%%%%%%%%%%%%%%%%%%%%%%%%
\setcounter{equation}{0}
\section{Relating relative and Virtual Fundamental classes }
\label{section8}
\medskip

  In this final section, we consider spaces  that are thin compactifications  {\em and} admit an  orientable implicit atlas.  Such a space $\ov{M}$ is compact, metrizable, and has both the fundamental class $[\ov M]$ defined in Section~1, and the Pardon virtual fundamental class $[\ov M]^{\mathrm{vir}}_A$ as in  \eqref{P.VFC}.  We give conditions under which these  correspond,  first for a single compact space in Proposition~\ref{Prop5.1}, and then for families in Theorem~\ref{5.maintheorem}.

  We will work with  rational coefficients, moving between Steenrod homology and dual \Cech cohomology using  the natural transformation 
\bear\label{5.tau}
\tau:  \sHH_k(X;\Q) \ma\longrightarrow^{\cong}  \cHH^{k}_c(X;\Q)^\vee
 \eear
(see \eqref{A.tau}) which is an isomorphism for any paracompact, locally compact space $X$. 
 Furthermore, if  $X$ admits an orientable implicit atlas $A$,  an orientation of $A$ (i.e. a nowhere vanishing section of $\mathfrak o_{X\; \mathrm{rel} \; \bd}$) trivializes the  orientation sheave , giving an identification
\bear\label{5.triv.or}
\cHH^{k}(X; \mathfrak o_{X\; \mathrm{rel} \; \bd})^\vee =     \cHH^{k}(X;\Q)^\vee.
\eear
  The orientation on $A$   also orients the regular locus $\ov{M}^{\mathrm{reg}}_\emptyset$. 

\smallskip

\begin{prop}\label{Prop5.1} 
Let  $M$ be an oriented manifold.  If a thin compactification $\ov{M}$ of $M$  admits
an  oriented implicit atlas $A$ such that $M \subseteq \ov{M}^{\mathrm{reg}}_\emptyset$  as oriented manifolds of same dimension $d$, then the fundamental class \eqref{1.fc}   corresponds  to Pardon's virtual fundamental class: 
\best
[\ov{M}] \, =\, [\ov{M}]_A^{\mathrm {vir}}
\eest
under the isomorphisms \eqref{5.tau}  and \eqref{5.triv.or} for $X=\ov M$.
\end{prop}

\begin{proof}  In this situation, there are inclusions $M\subseteq \ov{M}^{reg}_\emptyset \subseteq \ov{M}$ of open subsets, where $M$ and  $\ov{M}^{reg}_\emptyset$ are oriented paracompact $d$-manifolds, and $\ov{M}$ is compact. The corresponding Steenrod restriction maps \eqref{1.rhoU} induce a commutative diagram
$$
\xymatrix@=4mm{
& \sHH_d(\ov{M}^{reg}_\emptyset ; \Q) \ar[rd]^{\rho_M} & \\
\sHH_d(\ov{M};\Q) \ar[rr]^{\cong}_{\rho} \ar[ur]^{ \rho_{reg}} & & \sHH_d(M;\Q),\\
}
$$
where the bottom arrow is the isomorphism \eqref{YYFrseq}.  This maps via \eqref{5.tau} to  a corresponding diagram in the dual of rational \Cech cohomology,  in which  $\rho_{reg}$ is replaced by the \Cech map
$$
\rho_{reg}: \cHH^d(\ov{M}; \Q)^\vee \ \overset{\ \ \ }\longrightarrow\,   \cHH^d_c(\ov{M}^{reg}_\emptyset; \Q)^\vee,
$$
induced by the inclusion $\ov{M}^{reg}_\emptyset  \hookrightarrow \ov{M}$,     and thus corresponding to  \eqref{ourinclusionj}  with $X=\ov{M}$ under  the coefficient trivialization \eqref{5.triv.or}.
With these identifications,  the statement of Lemma~\ref{Pardon.L}  becomes
   \bear\label{ourinclusion8}
{\rho_{reg}}[\ov{M}]^\mathrm{vir}_{A}\ =\ [\ov{M}^{reg}_\emptyset] \qquad\mbox{in $\cHH^d_c(\ov{M}^{reg}_\emptyset; \Q)^\vee$}.
\eear

Successively applying  equations \eqref{1.fc},  \eqref{5.2MN} and  \eqref{ourinclusion8}, and then using the commutativity of the diagram, we obtain
$$
\rho[\ov{M}]\,=\, [M]\,=\, \rho_M[\ov{M}^{reg}_\emptyset]\,=\,\rho_M  \big({\rho_{reg}}[\ov{M}]^\mathrm{vir}_{A}\big)\,=\, \rho[\ov{M}]^\mathrm{vir}_{A}.
$$
The proposition follows because $\rho$ is an isomorphism. 
\end{proof} 

\medskip

 For  applications, one  needs a version of Proposition~\ref{Prop5.1} for families.  While Pardon does not explicitly describe implicit atlases for families, his applications (e.g. \cite[\S 9.3]{pardon}) show that  it is reasonable to consider  proper continuous maps
\bear
\label{5.Xfamily}
\xymatrix{
\X\ar[d]\\
\P
}
\eear
from a Hausdorff space to a Banach manifold that satisfy two conditions:
\smallskip
\begin{itemize}\setlength\itemsep{4pt}
\item[{\bf IA.1.}] Every fiber $\X_p$ of \eqref{5.Xfamily} admits  an oriented implicit atlas $A_p$. 
\item[{\bf IA.2.}]  For every path $\gamma$ in $\P$ from $p$ to $q$, there is an  oriented   implicit atlas  with boundary on $\X_\gamma$ which restricts to the chosen oriented implicit atlas on $\X_p$ and $\X_q$.
\end{itemize}
\smallskip
By (IA.1), each fiber has a virtual fundamental class 
$$   
[\X_p]^\mathrm{vir}_{A_p} \in \cHH^d(\X_p;\Q)^\vee,
$$
while  (IA.2) implies that  
$$ 
[\X_p]^\mathrm{vir}_{A_p} = [\X_q]^\mathrm{vir}_{A_q} \qquad \mbox{in}\quad     \cHH^d(\X_\gamma;\Q)^\vee
$$
(cf.  the proof of Lemma 9.3.2 of \cite{pardon}).   Applying Proposition~\ref{extLemma} we obtain:

\begin{lemma}
\label{5.extension}
For a family \eqref{5.Xfamily} that satisfies conditions (IA.1) and (IA.2), the association $p \mapsto [\X_p]^\mathrm{vir}_{A_p} $  extends uniquely to a relative homology functor $\mu^\mathrm{vir}_A  $.
\end{lemma}

\medskip

To extend Proposition~\ref{Prop5.1} to families,  we  again consider a relatively oriented Fredholm family \eqref{2.1} that admits a  metrizable, relatively thin compactification 
\bear
\label{5.lastM}
\xymatrix{
\ov\M \ar[d]^{\ov\pi} \\
\P.
}
\eear
We also assume that $\ov\M$  satisfies (IA.1) and (IA.2) above, with an inclusion
\bear\label{5.MinMreg}
\M_p\subseteq ({\ov\M_p})^{reg}_\emptyset \quad\mbox{ for all $p\in \P$}. 
\eear
In this case, both the relative fundamental class $[\ov\M]^\rfc$  and Pardon's virtual fundamental class $[\ov\M]^\mathrm{vir}_A$ are  defined, and both are relative homology functors.   The following theorem shows that they are equal.

\begin{theorem} 
\label{5.maintheorem}
Suppose that  a  metrizable, relatively thin compactification $\ov\pi:\ov \M\ra \P$  of a relatively oriented Fredholm family satisfies (IA.1), (IA.2) and \eqref{5.MinMreg}.   Then  the relative fundamental  class and  Pardon's virtual fundamental class are equal  as relative homology functors.  In    particular, 
\bear
\label{5.ssvfc}
[\ov\M_p]^\rfc= [\ov\M_p]^\mathrm{vir}_{A_p} 
\eear 
for all $p\in\P$.
\end{theorem}

\begin{proof} 
For each regular value $p$ of $\ov\pi$, $\ov\M_p$ is a  metrizable, relatively thin compactification of $\M_p$, and has an implicit atlas $A_p$  with  $\M_p\subset (\ov\M_p)^{reg}_\emptyset$.  Proposition~\ref{Prop5.1}  then shows that   $[\ov\M_p]^\rfc = [\ov\M_p]= [\ov\M_p]^\mathrm{vir}_{A_p}$.  But the set of regular values of $\ov\pi$ is dense in $\P$,  so Proposition~\ref{extLemma}(a) gives \eqref{5.ssvfc}.  
 \end{proof}

\begin{cor} 
\label{5.lastCor} Assume $\ov\pi:\ov \M\ra \P$ satisfies the assumptions of Theorem~\ref{5.maintheorem} and $\ov f:\ov \M\ra Y$ a continuous map. Then for each class $\beta\in \cHH^d(Y;\Q)$ the function 
\bear
\label{5.invariant}
I_\beta(p)\ =\ \lg f_*[\ov{\M}_p]^\rfc, \; \beta\rg \ =\ \lg f_*[\ov{\M}]^{\mathrm{vir}}_{A_p},\beta\rg 
\eear
is independent of $p$ on each path-connected component of $\P$.
 \end{cor} 

\begin{proof}
 Because each fiber $\ov\M_p$ is compact, the restriction of $f$ to $\ov\M_p$ is proper and continuous, so induces a map $f_*$ in \Cech homology.  Then \eqref{5.invariant} follows from \eqref{5.ssvfc} and \eqref{1.invts}.
 \end{proof}

%%%%%%%%%%%%%%%%%%%%%%%%%%%%%%%%%%%%%%%%%%%%%%%%
%%%%%%%%%%%%%%%%%%%%%%%%%%%%% Appendix A %%%%%%%%
%%%%%%%%%%%%%%%%%%%%%%%%%%%%%%%%%%%%%%%%%%%%%%%%
\setcounter{equation}{0}
\setcounter{section}{0}
\setcounter{theorem}{0}
\appendix
\section{Comparison of homology theories}\label{sectionA}
 \medskip

  This first appendix records needed facts about the Steenrod, \Cech and Borel-Moore homology and the corresponding cohomology theories, and provides references.

\subsection{Spaces.}   We will consider five categories:  the category ${\cal A}$ of Hausdorff spaces and continuous maps, the subcategories $\AC$ of compact spaces  and  $\ACM$ of compact metric spaces,  the category $\ALC$ of locally compact spaces and proper continuous maps, and the subcategory $\AEC\subset \ALC$ of   locally compact,  separable metric spaces with finite (covering) dimension.  Every $n$-dimensional separable metric space is homeomorphic to a subset of $\R^{2n+1}$ \cite[Theorem~V 3]{HW}, and using  local compactness one can  lift this  to a homeomorphism into $\R^{2n+2}$ with a {\em closed} image   \cite[IV.8.2 and 8.3]{D}.  Thus  objects in $\AEC$ can be regarded as closed subsets $X$ of euclidean space  (``euclidean closed'').

 For finite-dimensional Hausdorff manifolds,  the properties of being $\sigma$-compact and second countable  are equivalent, and any manifold with these properties is separable, paracompact, metrizable, and in the category $\AEC$.

\subsection{Homology.}  Steenrod homology $\sHH_*$ and \Cech homology $\cHH_*$  are introduced in Section~1.  For   the intersection theory done in Sections~4 and 5, it is useful to also use   Borel-Moore homology $\HBM_*$.  Throughout, we  restrict attention to constant coefficients in $R=\Z$ or $\Q$. 

\begin{itemize}\setlength{\itemsep}{4pt}
\item[] \hspace*{-10mm} $\bullet$\   Steenrod (also called Steenrod-Sitnikov) homology   is defined on $\ALC$ using chains that are dual to compactly supported, finite-value Alexander-Spanier cochains \cite[\S 4]{ma}.   It can also be defined  on locally compact metrizable spaces using ``canonical coverings'' \cite{Sk, Sk2};  see also Milnor's  construction on compact pairs   \cite{milnor}. It has an extension to  paracompact spaces in $\A$ called strong homology \cite[Chapter 19]{Mardesic}.

\item[] \hspace*{-10mm} $\bullet$\    \Cech homology is defined for pairs in $\A$ using nerves of covers.  It is only a partially exact homology theory, but it has the Continuity Property \eqref{1.Cech.ContinuityProperty} \cite[pages 260-261]{ES} and is exact for finitely triangulable spaces \cite[\S IX.9]{ES}. 

\item[] \hspace*{-10mm} $\bullet$\     Borel-Moore homology $\HBM_*$  is  defined on  $\ALC$  using sheaves \cite{BM, Br2, Iv} or on $\AEC$ using singular cohomology \cite[Chapter 19]{F1}, \cite[Appendix~B]{F2}.  It  has all of the properties listed  in Section~1 for Steenrod homology.
In  particular,  for each open set $U\subseteq X$ there is a natural  restriction map 
$$ 
 \rho_U: \HBM_*(X)\to \HBM_*(U)
$$
corresponding to \eqref{1.rhoU} and an exact sequence corresponding to \eqref{1.LES} (cf. \cite[IX.2.1]{Iv}). 
\end{itemize}

Steenrod and Borel-Moore   satisfy the Eilenberg-Steenrod axioms, so are  naturally isomorphic to singular homology on the category of triangulable spaces.  Both are single space theories  in the sense of \cite[\S X.7]{ES} (see page vii in \cite{ma} and Corollary~V.5.10 and the cautionary note  V.5.12 \cite{Br2}).   In particular,  the theorems in Section~X.7 of \cite{ES} show that
\bear
  \label{def.rho.cohom}
\sHH_*(X, A)= \sHH_*(X\Setminus A),\hspace{8mm} \HBM_*(X, A)= \HBM_*(X\Setminus A)
 \eear  
for any closed pair $(X,A)$. Steenrod  and Borel-Moore homologies satisfy Minor's modified  continuity property:  if $\{X_\alpha\}$ is an inverse system  in $\ACM$ with $X = \varprojlim X_\alpha$, then there is a natural  exact sequence
\bear\label{A.cont.seq}
\xymatrix@=3.5mm{
0 \longrightarrow {\varprojlim}^{\!1}\,  \sHH_{k+1}(X_\alpha) \longrightarrow \sHH_k(X) \longrightarrow \varprojlim\sHH_k(X_\al)\longrightarrow 0,
}
\eear
and a corresponding sequence with $\sHH_*$  replaced by $\HBM_*$ \cite[Thm.~4]{milnor}, \cite[p.81]{ma1}, \cite[V.5.15]{Br2}. 
There are many axiomatic characterizations of  Steenrod homology (cf. \cite{In, BMd}.  In particular, Theorem~1 of \cite{In} asserts:
\begin{itemize}
\item  Up to natural isomorphism,  $\sHH_*(X;G)$  is the unique  homology theory on $\AC$ that satisfies the Eilenberg-Steenrod axioms, is functorial in both $X$ and $G$, and has the  continuity property \eqref{1.Cech.ContinuityProperty} for   infinitely divisible abelian groups $G$.
\end{itemize}

\subsection{Natural Transformations.}  For any closed subset $X$ of  a paracompact $n$-dimensional oriented manifold $N$ and  $R=\Z$ or $\Q$,  there is a natural  isomorphism (a version of  Steenrod duality)
\bear\label{A.D1}
H^k(N, N\Setminus X;R) \overset\cong\longrightarrow \sHH_{n-k}(X;R)
\eear
 (see \cite[Theorem~11.15]{ma}, using the isomorphism  from the Alexander-Spanier to the singular  cohomology of  $(N, N\Setminus X)$  in the proof of \cite[Cor.~6.9.6]{Sp1}).  There is a similar  isomorphism (Poincar\'{e}-Alexander duality)
 \bear\label{A.D2}
H^k(N, N\Setminus X;R) \overset\cong\longrightarrow \HBM_{n-k}(X;R)
\eear
for Borel-Moore homology \cite[Thm.~10.4]{Sp2} or \cite[IX.4.7]{Iv}.   Taking $N=\R^n$,  \eqref{A.D1} and \eqref{A.D2} together show  that there is a natural isomorphism $\beta$ from Steenrod to Borel-Moore   homology   on the category $\AEC$.  In fact, on  $\AEC$ with $R=\Z$ or $\Q$, there is a commutative diagram of natural transformations
\begin{equation}
\label{A.trianglediagram}
\begin{gathered}
\xymatrix@=3.5mm{
 & \  \sHH_k(X;R)  \ar[rd]^\gamma \ar[dd]_\beta^\cong  && \\
H_k(X;R) \ar[ru]^\phi \ar[dr]^\psi &&  \ \  \cHH_k(X;R)   \\
& \ \quad \HBM_k(X;R)  \ar[ru]_\alpha  & \\
}
\end{gathered} \end{equation}
 where $\gamma$ is as in  \cite[Theorem~4]{Sk}, $\psi=\beta \circ \phi$ and   $\alpha= \gamma \circ \beta^{-1}$  
($\phi$ factors through compactly supported homology:  see  page 291 and the first arrow on page 308 in \cite{ma}.) Furthermore,   
\begin{enumerate}[(i)]\setlength{\itemsep}{4pt}
\item   The isomorphism   $\beta$  carries the restriction map    \eqref{1.rhoU} to the corresponding restriction map  in Borel-Moore homology, and carries the exact sequence \eqref{1.LES} to the corresponding   Borel-Moore sequence by matching both to the exact sequence 
\bear
\label{A.LES}
 \cdots \longrightarrow H^{k-1}(N\Setminus X) \overset{\delta}\longrightarrow H^k(N, N\Setminus X) \overset{j^*}\longrightarrow H^k(N)  \overset{\iota^*}\longrightarrow H^k(N\Setminus X) \longrightarrow \cdots
\eear
in Alexander-Spanier (or equivalently singular) cohomology;  cf. Theorems~11.15 and 8.1 in \cite{ma},  and diagram IX.3.5 in  \cite{Iv}.  
(Steenrod and Borel-Moore homologies are also isomorphic on the category $\ACM$ \cite[\S 5]{BM}.)

\item  $\gamma$ is surjective,  and is  an isomorphism if  $R=\Q$, or if $R=\Z$  and either (i) $k=\dim X$  or (ii) $X$ is a compact  and   locally contractible \cite[Thm.~4]{Sk}.   
\item $\phi$ and $\psi$ are isomorphisms if   $X$ is  compact and locally contractible \cite[\S 9.6]{ma}, \cite[\S V.12]{Br2}. 
\item  $\alpha$ is an isomorphism for  locally contractible spaces $X$ in $\ACM$  \cite[V.5.19]{Br2}. 
\end{enumerate}
 In particular,  the natural transformations in \eqref{A.trianglediagram} give isomorphisms
 \bear\label{A.6}
 \HBM_*(X; \Q)\ \cong\   \sHH_*(X; \Q)\ \cong\ \cHH_*(X; \Q)
\eear
for compact metric spaces $X$,  and  
\bear\label{A.H_*N}
\xymatrix{
H_*(N; \Z) \ar[r]_\cong^\phi  & \sHH_*(N;\Z) \ar[r]_\cong^\gamma  & \cHH_*(N;\Z)
}
\eear
 for compact manifolds $N$ and for finite polyhedra.  
 
 \subsection{Cohomology.} \label{sSA.cohom}
  Singular,  \Cech   and Alexander-Spanier cohomology are all defined on the category  of pairs in ${\cal A}$; we will restrict attention to paracompact pairs and constant coefficients $R=\Z$ or $\Q$. The last two are  naturally isomorphic 
\cite[6.8.8 and Exercise 6.D.3]{Sp1}, so we use the notation $\cHH^*$ for both.  The corresponding theories based on compactly supported cochains are also isomorphic, and will be denoted by $\cHH^*_c$.    In Part~II of \cite{ma}, Massey also uses  locally finite-valued  cochains for closed pairs, but these give the same theory as Alexander-Spanier cohomology by  \cite[Thm. 8.1]{ma}.   Thus for paracompact pairs $(X,A)$ and $R=\Z$ or $\Q$, it suffices to consider   three theories:  $\cHH^*$, $\cHH^*_c$ and  singular cohomology $H^*$.  These  are related by natural transformations
\bear\label{A.New2}
\xymatrix{
\cHH^*_c(X, A;R) \ar[r]^{\mu} & \cHH^*(X,A;R) \ar[r]^{\nu}   &  H^*(X,A;R)
}
\eear
which preserve cup products  \cite[Thm. III.2.1 and Cor. III.4.12]{Br2} and \cite[\S 6.5, \S 6.9]{Sp1}.   Furthermore, 
\begin{enumerate}[(i)]\setlength{\itemsep}{4pt}
\item $\mu$ is an isomorphism if $X$ is compact \cite[6.6.9]{Sp1}.  %{\cred careful Cech vs AS here} 
\item  $\nu$ is an  isomorphism if $X$ and $A$ are locally contractible (e.g.  manifolds) and $A$ is either closed  \cite[Theorem III.2.1]{Br2} or  open  (proof of 6.9.6 in  \cite{Sp1}).
\end{enumerate}

 For paracompact,  locally compact spaces $X$, Steenrod homology pairs with $\cHH^*_c$ in the sense that there is a Kronecker pairing 
\bear\label{A.Kronecker1}
\sHH_*(X)\otimes \cHH_c^*(X) \to R
\eear
(cf. \cite[\S 4.8]{ma}), and an exact   sequence
\bear\label{A.UCT}
 0 \to \mbox{Ext}\left(\cHH^{k+1}_c(X), R\right) \longrightarrow \sHH_k(X;R)\overset{\tau}\longrightarrow \mbox{Hom}\left(\cHH^{k}_c(X), R\right) \to 0
\eear
 that is natural in both $X$ and $R$ \cite[Cor. 4.18 and p.371]{ma}.   The same sequence holds with $\sHH_*$ replaced by $\HBM_*$  \cite[Theorem 3.3]{BM}.    In particular, there is a natural transformation
 \bear\label{A.tau}
\tau:  \sHH_k(X;R) \longrightarrow  \cHH^{k}_c(X;R)^\vee
 \eear
 where ${ }^\vee$ denotes dual, i.e. $\mbox{Hom}(\cdot, R)$.  By \eqref{A.UCT}, $\tau$ is an isomorphism for $R=\Q$.   For compact metric spaces, it factors through \Cech homology, as follows.

\begin{lemma}\label{A.taulemma}
On the category $\ACM$ with $R=\Z$ or $\Q$, $ \cHH_c^*(X)= \cHH^*(X)$ and  there is a commutative diagram of natural transformations
\bear\label{A.tauDiagram}
\xymatrix@=3mm{
 \sHH_*(X)   \ar[dr]_\gamma  \ar[rr]^{\tau}  &&  \cHH^*(X)^\vee \\
& \cHH_*(X) \ar[ur]_{\sigma}& 
}
\eear
with  $\gamma$   as in \eqref{A.trianglediagram} and $\sigma$ as defined below.  For $R=\Q$, all three maps are isomorphisms.
\end{lemma}
\begin{proof} 
First note that the compact metric space $X$ can be written as the inverse limit of a system $\{X_{\ell}\}$ in the category of finite polyhedra  (cf. \cite[{p.82}]{ma1} and \cite[6.6.7]{Sp1}). 
The homology maps induced by the inclusions   $X\to X_\ell$ define maps $\iota_*, j_*$ and $\check\iota_*$ into inverse systems  as in the diagram below.   But $\gamma$ is an isomorphism for finite polyhedra, thus we get a diagram \eqref{A.tauDiagram} when $X$ is replaced by $X_\ell$, and $\si$ is defined to be $\tau\circ\gamma^{-1}$.    By the naturality of $\gamma$ and $\tau$, there are induced maps $\gamma_{\bullet}$, $\tau_{\bullet}$  and  $\si_{\bullet}$ between the inverse systems, and the  front left and back squares in the diagram commute.  Furthermore, $\gamma_\bullet$ is an isomorphism, and the bottom triangle commutes. 

\item 
\begin{minipage}[t]{3.1in}
 The map $\check\iota_*$ is an isomorphism by \Cech continuity,  and $j_*$ is an isomorphism   by the continuity of \Cech  cohomology  and the fact that
  $\varprojlim\Hom(\cdot, R)=\Hom(\varinjlim \cdot, R)$.   The first statement of the lemma follows by defining $\sigma$ to be the composition  $\check \iota_*^{-1} \si_{\bullet}j_*$.  

\indent For $R=\Q$,  the $\lim^{\! 1}$ term in \eqref{A.cont.seq} and the  $\mbox{Ext}$ term in \eqref{A.UCT} vanish.  Hence  $\iota_*$ and  $\tau$, and therefore $\gamma$ and $\sigma$,   are isomorphisms   (cf. \cite[Remark 5.0.2]{pardon}). \end{minipage}
 \begin{minipage}[t]{3in}
\vspace*{-2mm }
\begin{center} \hspace*{5mm} \xymatrix@=3mm{
 \sHH_*(X)  \ar[dd]_{\iota_*} \ar[dr]_\gamma  \ar[rr]^\tau   &&  \cHH^*(X)^\vee \ar[dd]^{\check\iota_*}_\cong\\
& \cHH_*(X)  \ar@{-->}[ur]_{\sigma}  \ar[dd]_(.2){j_*}^(.2)\cong& \\
\varprojlim \sHH_*(X_\ell)  \ar[dr]_{\gamma_{\bullet}}^{\cong} \ar[rr]^(.4){\tau_{\bullet}} |\hole  && \varprojlim \cHH^*(X_\ell)^\vee  \\
&\varprojlim \cHH_*(X_\ell) \ar@{-->}[ur]_{\si_{\bullet}} &  
}
\end{center}
\end{minipage}

\end{proof}

Thus on $\ACM$,    $\sigma$ defines a  Kronecker   pairing
$\cHH_*(X)\otimes \cHH^*(X) \to R$
in \Cech theory,  and this  is related to \eqref{A.Kronecker1} by  
\bear\label{A.twopairings}
\langle \gamma(b),\, \alpha\rangle\,=\, \langle b,\, \alpha \rangle  \qquad  \forall b\in \sHH_*(X), \ \alpha\in\cHH^*(X).
\eear

 \subsection{ Borel-Moore  via embeddings.}  Fulton and MacPherson developed a simplified version of Borel-Moore homology on $\AEC$  for use in algebraic geometry; see Chapter~19 of \cite{F1} or Appendix~B of \cite{F2}.  Fixing coefficients in a field or in $\Z$, one {\em defines} the  embedded Borel-Moore homology   of a closed subset $X$ of $\R^m$ by  formula \eqref{A.D2}:
\bear\label{Aembedded1}
\ov{H}_k(X)\  \overset{\rm def}=\ \  H^{m-k}(\R^m,\R^m\Setminus X),
\eear 
 (cf.  \cite[\S19.1, eq (1)]{F1}). Similarly,  for any open subset $U$ of $X$, one sets
\best
\ov{H}_k(U)\  =\  \ov{H}_k(X, A) \ =\ H^{m-k}(\R^m\Setminus A,\R^m\Setminus X).
\eest
where $A=X\Setminus U$  (cf. \cite[(30) in Appendix B]{F2}).   Using facts about singular cohomology, one can  verify that the groups $\ov{H}_*$ have all of the properties of Borel-Moore homology (cf.  \cite{F1}, Section~19.1 and especially Example~19.1.1).   For a direct correspondence,  note that   for an object $X$ in $\AEC$, the choice of a closed embedding $\iota:X\to \R^m$ with image $X^e$ induces natural isomorphisms
\bear
\xymatrix@=8mm{
    \HBM_*(X)  \ar[r]_{\iota_*\qquad\quad }^{\cong\qquad \quad}     &     \HBM_*(X^e)\ \cong\ \ov{H}_*(X^e),  }
\eear
 where the second   is the composition of \eqref{A.D2} and \eqref{Aembedded1}.

\subsection{Cap Products.} 

For any closed subset $X$ of a locally compact $Z$, we have a localized (sheaf theoretic supported) cap product 
\bear\label{A3.1}
\HBM_n(Z) \otimes H^k(Z, Z\Setminus X)   \xra{\cap}\HBM_{n-k}(X), 
\eear
on Borel-Moore homology with coefficients in any commutative Noetherian ring, cf.  \cite[IX.3.1, \S IX.5 and \S II.9]{Iv}, \cite[\S 19.1]{F1}. 
This  has several naturality properties  (cf. \cite[\S IX.3]{Iv}),  including:

\begin{itemize}\setlength{\itemsep}{4pt}

\item  for a proper map $f:(Z', X')\to (Z, X)$ of  closed, locally compact pairs, 
\bear\label{A.cupnaturality1}
f_*(a' \cap f^*\xi)= f_*a' \cap \xi. 
\eear
\item for closed subsets $X\ma\hookrightarrow^i Z$ and $Y\subseteq Z$, 
\bear\label{A.cupnaturality2}
(a \cap \xi)\cap i^*\eta= a\cap (\xi \cup \eta) \quad \mbox{in }\HBM_*(X\cap Y).
\eear
\item    if $X\subseteq Z$ is closed  and $U\ma\hookrightarrow^j Z$ is open, then the restriction to $U$ satisfies 
\bear\label{A.cupnaturality3}
\rho_{U\cap X}(a \cap \xi)= \rho_U(a)\cap j^*\xi  \quad \mbox{in }\HBM_*(U\cap X).
\eear
\end{itemize}
\vspace{5pt}

\noindent These hold for all  $a \in \HBM_*(Z)$, $a' \in \HBM_*(Z')$, $\xi\in H^*(Z, Z\Setminus X)$ and $\eta\in H^*(Z, Z\Setminus Y)$. \subsection{Fundamental classes.} 
Let $N$ be an oriented  topological $n$-manifold (a Hausdorff space locally homeomorphic to $\R^n$).  Its orientation determines  fundamental classes
\bear\label{fund.BM.sheaf}
[N]\in \sHH_n(N;\Z) \hspace{6mm}  [N]\in \HBM_n(N;\Z)
\eear 
in Steenrod  \cite[\S 4.9]{ma}  and Borel-Moore  homology  \cite[IX.4.6]{Iv}.   The restriction of $[N]$   to an  open set $U\subseteq N$ is the fundamental class of $U$:
$$
\rho_U[N]=[U].
$$
On each component $N_\alpha$ of $N$,  the orientation determines isomorphisms $\sHH_n(N_\alpha;\Z) = \Z = \HBM_n(N_\alpha;\Z)$ under which $[N_\alpha]$ corresponds to $1$.  The naturality of 
the  transformation $\beta$ in \eqref{A.trianglediagram} with respect to restriction maps then implies  that  the two fundamental classes \eqref{fund.BM.sheaf} correspond under $\beta$.

 For any closed subset $Y$ of $N$, the cap product \eqref{A3.1} with the fundamental class is an isomorphism  
 \bear\label{PD.BM}
D: H^{k} (N, N \Setminus Y) \underset\cong{\xrightarrow{  [N]\cap \;\;}}\HBM_{n-k}(Y)
\eear
which is precisely  \eqref{A.D2}  cf.  \cite[IX.4.7]{Iv}.  There is corresponding isomorphism with values in Steenrod homology.  In particular, if $N$ is compact, taking $Y=N$ gives the Poincar\'{e} duality isomorphisms
\bear\label{A.PD}
D: H^k(N) \overset\cong \longrightarrow \HBM_{n-k}(N)  \hspace{8mm}  D: H^k(N) \overset\cong\longrightarrow \sHH_{n-k}(N).
\eear

\subsection{Submanifolds.}  Suppose that $N$ is a  paracompact  oriented $C^1$ $n$-manifold and 
$$
V\hookrightarrow N
$$
 is  a properly embedded oriented  submanifold of codimension $k$.   By identifying a neighborhood $U$ of $V$ in $N$ with the total space of the normal bundle to $V$ and using excision, the Thom class  of  the normal bundle defines a singular cohomology  class
\bear\label{A.transverseclass}
u=u_{V, N}\in H^{k}(U, U\Setminus V) = H^{k}(N, N\Setminus V) 
\eear
Following Fulton \cite[\S 19.2]{F1}, we call $u$ the {\em orientation class} of $V$ in $N$.  Then  $[V]$ and $[N]$ are related 
\bear\label{A.[XV]}
[V]\ =\ [N] \cap u_{V, N} \quad \mbox { in } \HBM_{n-k}(V)
\eear
as in \cite[IX.4.9]{Iv}.  In particular, $[V]$ corresponds to $u_{V,N}$ under the duality \eqref{PD.BM}. The naturality of Thom class gives a naturality property of $u$:  if a $C^1$ map $f:M\to N$ of oriented manifolds is transverse to  $V$, then $W=f^{-1}(V)$ is an oriented submanifold of $M$ with orientation class 
\bear\label{A.nat.Thom}
u_{W,M} = f^*u_{V,N}.
\eear

\subsection{Intersection Pairing.} For   closed subsets $X$ and $Y$ of a manifold $N$ as in \S A.8, there is a cup product in singular cohomology
$$
 H^{n-k}(N, N\Setminus X) \otimes H^{n-\ell}(N, N\Setminus Y) \overset{\cup}\longrightarrow H^{2n-k-\ell}(N, N\Setminus(X\cap Y))
$$
 (cf. \cite[\S 5.6]{Sp1}, noting that $\{N-X, N-Y\}$ is an excisive pair by \cite[4.6.4]{Sp1}). The duality \eqref{PD.BM} translates the cup product   into the cap product 
  \bear\label{BM.cap}
\HBM_k(X)  \otimes H^{n-\ell}(N, N\Setminus Y) \overset{\cap }\longrightarrow  \HBM_{k+\ell-n} (X\cap Y)  
\eear
 by the formula:
\best
b \cap \al  = D\big( D^{-1}b \cup \al \big).
\eest

 Note that  the righthand side is $b\cap i^*\alpha$ for the cap product  \eqref{A3.1}, where $i$ is the inclusion of $X$ into $N$ (cf.    \eqref{A.cupnaturality2} with $a=[N]$,  $b=D\xi=[N]\cap \xi$, and $\eta=i^*\alpha$).  Thus   this cap product depends only on the restriction of $\alpha$ to $X$.

  Applying  \eqref{PD.BM} again yields  a natural intersection pairing in Borel-Moore homology 
\bear\label{A.BM.fat.int}
 \HBM_k(X) \otimes \HBM_\ell(Y) \overset{\bullet }\longrightarrow  \HBM_{k+\ell-n} (X\cap Y)
\eear
given by
\bear\label{A.22b}
a \bullet b \, =\, D\big(D^{-1}b\cup D^{-1}a\big) = b\cap  D^{-1}a.
\eear
(The order reversal is needed to obtain the correct signs,  cf. \cite[\S V.11]{Br2}). 

The intersection pairing \eqref{A.BM.fat.int} is natural with respect to the restriction map $\rho_U$ to  any open subset $U$ of $N$:   the naturality of the cup product and  \eqref{def.rho.cohom} translates into the identity
\bear\label{rho.dot}
\rho_U (a\bullet b)= \rho_U( a )\bullet \rho_U (b).
\eear
%(see \cite[\S V.11, (69)]{Br2}). 
 In particular, if $U$ is any neighborhood of $X\cap Y$, then  the restriction to $U$ induces the identity on $H^{BM}_*(X\cap Y)$ and hence 
 \bear\label{A.24B}
 a\bullet b=  \rho_U( a )\bullet \rho_U (b). 
\eear  Thus the intersection 
 localizes on any open neighborhood of $X\cap Y$.

\begin{ex}
\label{A.intersectionlemma1}
Let $X$ and $Y$ be   properly embedded oriented submanifolds  of  an oriented  $C^1$ manifold $N$.  If $X$ and $Y$  intersect transversally, then $X\cap Y$ is an oriented manifold, and 
$$
[X]\bullet [Y]\ =\ [X\cap Y] \qquad \mbox{ in } \HBM_*(X \cap Y). 
$$
The proof  exactly as in the proof of Theorem VI-11.9 of \cite{Br1}, using formulas   \eqref{A.cupnaturality2}, \eqref{A.[XV]},   \eqref{A.22b}  and the naturality of Thom classes, and interpreting  all terms  as elements of Borel-Moore homology.
\end{ex}

\medskip

More generally,  consider proper maps $f:M\ra N$ and $g:P \ra N$  between oriented manifolds.  Set 
\bear\label{Ap.D.Z}
Z\ =\ \big\{ (x,y)\in N\ti N\,\big|\,  f(x)=g(y)\big\}
\eear
and define  $h:Z\ra N$ by $h(x, y)=f(x)=g(y)$.

 \begin{lemma}\label{L.f.transv.g}  Suppose that maps $f$ and $g$ as above are  transverse  and have complementary dimensions in $N$. Then $Z$ is a 0-dimensional manifold, and has an induced orientation  such that
 \bear\label{dot.fg}
f_*[M]\bullet  g_*[P] \ =\    h_*[Z]
\eear
in $\HBM_0(X\cap Y;\Z)$, where $X$ and $Y$ are the images of $f$ and $g$.
\end{lemma}

\begin{proof} 
The assumptions that  $f$ and $g$ are transverse with complementary dimension imply that $Z$ is a discrete set of points and that $f$, $g$, and $f\ti g$ are  immersions
 at each point $z=(x,y)\in Z$.    Hence we can find disjoint open neighborhoods $U_p$ of the points $p\in X\cap Y$ so that, for 
$U=\ma\bigsqcup  U_p$, we have
\best
f^{-1}(U)= \bigsqcup_{x\in f^{-1}(X\cap Y)} V_x, \quad \mbox {and} \quad g^{-1}(U)= \bigsqcup_{y\in g^{-1}(X\cap Y)} W_y
\eest 
where $\{V_x\}$ (resp. $\{W_y\}$) are disjoint neighborhoods of $x$ in $M$ (resp. $y$ in $P$), and where  $f$ and  $g$ restrict to a proper embeddings $f:V_x \ra U$ and  $g:W_y \ra U$.

   As in  \eqref{A.24B},  the lefthand side of \eqref{dot.fg}  localizes, and hence is a locally finite sum of local intersections 
 $$
 f_*[M]\bullet  g_*[P] \ =\  \sum_{(x, y)\in Z} f_*[V_x] \bullet g_*[W_y]\ =\  \sum_{(x, y)\in Z} \left[f(V_x) \cap g(W_y)\right], 
 $$
where the second equality is obtained by  applying  Example~\ref{A.intersectionlemma1}.   But $f(V_x) \cap g(W_y)$ is exactly $h(z)$, and is oriented as the local intersection of $f$ and $g$. This identification induces an orientation 
$[z]\in \HBM_0(z;\Z)$ for each $z\in Z$, and hence an orientation on $Z$.   The lemma follows.
\end{proof}

\subsection{Intersection Numbers.} With constant coefficients  $R=\Z$ or $\Q$, the intersection of classes   $a\in \HBM_k(X;R)$ and $b\in \HBM_\ell(Y;R)$  of complementary dimension (i.e. $k+\ell =n=\dim N$) is a  0-dimensional class
$$
a\bullet b\in \HBM_0 (X\cap Y;R).
$$
If $X\cap Y$ is compact,  there is an  augmentation map 
$
 \ep: \HBM_*(X\cap Y;R)\to R
$
 (induced by the map from $X\cap Y$ to a point), and the intersection number is  defined by
\bear\label{int.number} 
a\cdot b = \ep (a\bullet b) \in R. 
\eear
This  can be written in many other ways:
\bear\label{int.ep} 
a\cdot b \ =\  \ep (b \cap \al ) \ =\  \lg  b\cap \al ,1 \rg \ =\   \lg b, \al \rg \ =\   \int_{b}\al,
\eear
for elements $\al=D^{-1}a$ of $H^*(N, N\Setminus X)$.  In particular, when $X=N$ is a compact oriented manifold, then $\al \in \cHH_c^*(N)$ is  Poincar\'{e} dual to $a$ under \eqref{A.PD}    (cf. \cite[V.11]{Br2}).

\medskip

\begin{ex}
\label{A.Ex3}   
Let $f$ and $g$ be proper maps as in Lemma~\ref{L.f.transv.g}, and assume that the   intersection of their images are compact.   Then   \eqref{Ap.D.Z} is a compact 0-dimensional oriented manifold,  consisting of finitely many points $x$ with sign $\ep(x)=\pm 1$.
 In this case, using \eqref{dot.fg}, the intersection number is 
 $$   
f_*[M]\cdot g_*[P]  \ =\ \ep(f_*[M]\bullet  g_*[P])\ =\ \ep( h_*[Z]) \ =\ \sum_{x\in Z} \ep(x) \in \Z. 
$$

\end{ex}

%%%%%%%%%%%%%%%%%%%
%%%%%%%%%%%%%%%%%%
%%%%%%%%%%%%%%%%%%%%%

\setcounter{equation}{0}
\setcounter{theorem}{0}
\section{Dimension theory}
\medskip

One can define the dimension of a topological space $X$ in  several  ways:
\begin{enumerate}\setlength{\itemsep}{4pt}
\item The Lebesgue covering dimension $\dim X$  is the smallest number $d$ so that every open cover has a refinement such that every $x\in X$ lies in at most $d+1$ sets of the refinement.  
\item The (large) cohomological dimension ${\mathrm{dim}}_2 X$ is the largest $k$ such that $\cHH^k(X, A;\Z)$ is non-zero for some closed set $A$ in $X$.  
%\item  The \Cech cohomological dimension ${\mathrm{dim}}_3 X$ is the smallest $d$ such that $\cHH^k(X; \Z)=0$ for all $k>d$.
%\item  The Steenrod homological dimension ${\mathrm{dim}}_4 X$ is the smallest $d$ such that $\sHH_k(X; \Z)=0$ for all $k>d$.
\end{enumerate}
In addition, if $X$ is a metric space, one has:
\begin{enumerate}
\item[(3)]  The Hausdorff dimension ${\mathrm{dim}}_H  X$ of  X is the infimum of  $\delta\ge 0$ with the following property: For any $\ep>0$, 
$X$ can be covered by  countably many sets $\{A_n\}$ with $\diam(A_n)< \ep$ and with $\sum_i \diam(A_n)^\delta <\ep$. 
\end{enumerate}

Standard results of dimension theory show that 
\begin{equation}\label{B.1}
 {\mathrm{dim}}_2 X\ =\  \dim X \qquad\mbox{and}\qquad \dim X\ \le\ {\mathrm{dim}}_H  X,
\end{equation}
where the first equality holds for all non-empty paracompact Hausdorff spaces \cite[36-15 and 37-7]{Na}, and the second holds for all separable metric spaces \cite[p. 107]{HW}.
These are related to condition~\eqref{1.1} as follows. 

\begin{lemma}
\label{LemmaB1}
If $X$ is a compact Hausdorff space with $\dim X=d$, then 
\bear\label{B.2}
\sHH_k(X;\Z)\ =\ \HBM_k(X;\Z)\ =\ 0 \ \quad \mbox{for all $k> d$.} 
\eear
\end{lemma}
\begin{proof}
For compact $X$ we have $\cHH^*(X)=\cHH_c^*(X)$.  
Using \eqref{B.1} and  taking $A$ to be the empty set in the definition of ${\mathrm{dim}}_2 X$, one sees that the condition $\dim X=d$ implies that $\cHH^k(X)=0$ for all $k>d$.  Then \eqref{B.2} follows from  \eqref{A.UCT} and the corresponding sequence  for 
Borel-Moore homology. 
\end{proof}

\begin{lemma}
\label{DimensionTheoryLemma}
Suppose that $X$ is a metric space and $M$ is a separable metrizable topological $d$-manifold.  
\begin{enumerate}[(a)]\setlength{\itemsep}{4pt}
\item  For any subspace $S\subseteq X$, $ \dim  S \le  \dim  X$ and $ {\mathrm{dim}}_H S \le  {\mathrm{dim}}_H X$.  
\item   If  $X$ is a countable union of closed subsets $X_i$, then $\dim  X \le \sup\, \dim  X_i$ and ${\mathrm{dim}}_H X \le \sup\, {\mathrm{dim}}_H X_i$. 
 
$$   
\dim  f(X)\, \le\, \dim  X.
$$
\item  $\dim M = {\mathrm{dim}}_H M=d$.
\item  If  $\ov{M}=M\cup S$ is a thin compactification    of $M$ with $\dim S\le d-2$,  then $\dim \ov{M}=d$.   
\end{enumerate}
 \end{lemma}

\begin{proof}
For Hausdorff dimension, (a), (b) and  (c) follow easily from the definition.  For   covering dimension,   (a) and (b)   are  Theorems~4.1.7 and  4.1.9  of \cite{E}, respectively, and  (c) is Corollary~1 of Section IV.4 of \cite{HW}.  Finally, note that $M$ can be written as a countable union of closed subsets $B_i$ (closed balls of integer radius in some metric) each with $\dim  B_i=d$.  Applying (b) to $\ov{M}= S\cup \bigcup B_i$ with $\dim S\le d-2$  shows that $\dim  \ov{M}=d$.  
\end{proof}

\medskip

  In general, a continuous map can increase covering dimension, as occurs for Peano's space-filling curve.  One can avoid such pathologies by working with Hausdorff dimension and Lipschitz maps.
Recall that a map $f:X\to Y$ between metric spaces is called Lipschitz if there is a constant $C>0$ such that
$$
\dist(f(x), f(y))\ \le\ C\, \dist(x, y)
$$
for all $x, y\in X$, and is {\em locally Lipschitz} if every point in $X$ has a neighborhood in which $f$ is Lipschitz.       Note that if $f$ is  locally Lipschitz then its restriction to  a compact set $K\subset X$ is Lipschitz.  Also note  that   the definition of Hausdorff dimension implies that if $f$ is  Lipschitz then 
\bear\label{B.LipHaus}
{\mathrm{dim}}_H f(X)\ \le \ {\mathrm{dim}}_H X.
\eear

\begin{lemma}
\label{LemmaA2}
Suppose that a subset $S$ of a metric space $Y$ is contained in the union of the images of a countable collection of    maps $\phi_n:U_n\to Y$, where  
each $U_n$ is  a $\sigma$-compact topological manifold of dimension $\le d$, and either
\begin{itemize}\setlength{\itemsep}{4pt}
\item[(i)]    $\phi_n$ continuous and locally injective, or 
\item[(ii)]   $\phi_n$ is locally Lipschitz, or
\item[(iii)]  $\phi_n$ is a $C^1$ map between $C^1$ manifolds.
\end{itemize}     
 Then $\dim  S \le d$ and, if $S$ is compact, 
$\sHH_k(S) = 0$ for all $k> d$.
\end{lemma}

\begin{proof}
Each $U_n$ is $\sigma$-compact, so can be covered by a countable collection of open sets  $\{B_{mn}\}$ with compact closures $\ov{B}_{mn}$.  Then $\phi_n$  restricts to maps $\phi_{mn}:\ov{B}_{mn}\to Y$.

In case (i),  $\phi_n$ is locally injective, so we can assume, after refining the cover $\{B_{mn}\}$, that each $\phi_{mn}$ is injective.    Then  
$\phi_{mn}$ is    a continuous injection from a compact set to a Hausdorff space, therefore  a homeomorphism onto its image.   Since covering dimension is a homeomorphism invariant, parts (a) and (c) of  Lemma~\ref{DimensionTheoryLemma}   show that   the image satisfies

$$
\dim \, \phi_{mn}(\ov{B}_{mn})\  =\ \dim \,  \ov{B}_{mn}\ \le\ d.
$$

In  case (ii), the assumptions also  imply that $U_n$ is metrizable.  After fixing a metric,   inequalities \eqref{B.1},  \eqref{B.LipHaus} and   Lemma~\ref{DimensionTheoryLemma}(a,c) imply that 
$$
\dim \, \phi_{mn}(\ov{B}_{mn})\ \le\ {\mathrm{dim}}_H\, \phi_{mn}(\ov{B}_{mn})\ \le\ {\mathrm{dim}}_H\,  \ov{B}_{mn}\ \le\ d.
$$
This  inequality also holds in case (iii) because   any $C^1$ map is locally Lipschitz. 

 In either case, we can apply  parts (a) and (b) of Lemma~\ref{DimensionTheoryLemma} to conclude that
$$
\dim  S \ \le \ \dim \  \bigcup_{m,n} \phi_{mn}(\ov{B}_{mn}) \ \le \ \sup_{m,n}\dim  \, \phi_{mn}(\ov{B}_{mn})\ \le\  d.
$$
The  lemma then follows by applying Lemma~\ref{LemmaB1}. 
\end{proof}

%%%%%%%%%%%%%%%%%%%%%%%%%%%%%%%%%%%%%%%%%%%%%%%%%%%%%%%%%%%%%%%%%%%%%%%%%%%%%%%

% \vspace{8mm}
%%%%%%%%%%%%%%%%%%%%%%%%%%%%%%%%%%%%%%%%%%%%%%%%%%%%%%%%%%%%%%%%%%%%%%%%%%%%%%%


{\small

\begin{thebibliography}{}

 
 \bibitem[Br1]{Br1} G.E.~Bredon, {\em Topology and  Geometry},  Springer-Verlag, 1993.
 
 
  \bibitem[Br2]{Br2} G.E.~Bredon, {\em Sheaf Theory}, second ed., McGraw-Hill, New York, 1997.
 
 
 
 \bibitem[BM]{BM}  A.~Borel and J.C.~Moore, {\em Homology theory for locally compact spaces}, Michigan Math. J. {\bf 7} (1960), 137-160.

 \bibitem[BeM]{BMd}  A.~Beridze  and L.~Mdzinarishvili, {\em On the axiomatic systems of Steenrod homology theory of compact spaces},  arXiv:1703.05070. 


 
 
 \bibitem[D]{D} A.~Dold,    {\em Lectures on algebraic topology},  Springer-Verlag, Berlin, 1972.
 
 
\bibitem[E]{E} R.~Engelking,    {\em Theory of Dimensions Finite and Infinite},  Heldermann Verlag, Lemgo Germany, 1995.



\bibitem[ES]{ES} S.~Eilenberg and  N.E.~Steenrod, {\em Foundations of algebraic topology}, Princeton University Press, Princeton, 1952.


\bibitem[F1]{F1} W.~Fulton, {\em Intersection Theory}, Springer-Verlag, Berlin, 1984.

\bibitem[F2]{F2} W.~Fulton, {\em Young Tableaux}, Cambridge University Press, Cambridge, U.K., 1997.


 
\bibitem[HW]{HW}  W.~Hurewicz and H.~Wallman, {\em Dimension theory}, Princeton University Press, Princeton, N.J., 1941.



 \bibitem[In]{In} H.~Inassaridze, {\em On the Steenrod homology theory of compact spaces}, Michigan Math. J. {\bf 38} (1991), 323-338.


\bibitem[Iv]{Iv}  B.~Iverson, {\em Cohomology of Sheaves}, Springer-Verlag, 1986.


 
\bibitem[IP1]{IPThin} E.N.~Ionel and T.H.~Parker, {\em   Thin compactifications and relative fundamental classes},  arXiv:1512.07894.  To appear in J. Symplectic Geometry.
 
  \bibitem[Mar]{Mardesic}  S.~Marde\v{s}i\'{c},   {\em Strong shape and homology}, Springer-Verlag, Berlin, 2000.

  
  
 
\bibitem[Mas]{ma}  W.S.~Massey,   {\em Homology and cohomology theory}, Marcel Dekker, New York, 1978.

 \bibitem[Mas2]{ma1}  W.S.~Massey, {\em How to give an exposition of the \v{C}ech-Alexander-Spanier type homology theory},   American Math. Monthly {\bf 85}  (1978), 75-83. 
   
\bibitem[Mil]{milnor}  J.~Milnor,   {\em On the Steenrod homology theory},  Novikov Conjectures, Index Theorems and Rigidity, Vol. 1 (Oberwolfach, 1993), 79-96, London Math. Soc. Lecture Note Ser., {\bf 226}, Cambridge Univ. Press, Cambridge, 1995.

 
 
 \bibitem[MS]{ms2} D.~McDuff and  D.~Salamon,    {\em J-holomorphic curves and symplectic topology},  American Mathematical Society Colloquium Publications, {\bf 52},  AMS, Providence, RI, 2004.

 \bibitem[Na]{Na} K.~Nagami,    {\em Dimension theory},  Academic Press, , New York, 1970.


 

\bibitem[Pd]{pardon} J.~Pardon, {\em An algebraic approach to virtual fundamental cycles on moduli spaces of pseudo-holomorphic curves}, Geometry \& Topology {\bf 20} (2016), 779-1034.

 
 
\bibitem[RT1]{rt1} Y. Ruan and G. Tian, {\em A mathematical theory of quantum cohomology}, J. Diff. Geom. {\bf 42} (1995), 259-367.

 
\bibitem[Sk]{Sk} E.G.~Sklyarenko, {\em  Homology theory and the exactness axiom}, Russian Math. Surveys, {\bf 24} Issue 5,, (1969), 91-142.

\bibitem[Sk2]{Sk2} E.G.~Sklyarenko, {\em  On homology theory associated with the Aleksandrov-\Cech cohomology}, Russian Math. Surveys, {\bf 34} Issue 6, (1979), 103-137.


\bibitem[Sp1]{Sp1} E.H.~Spanier, {\em Algebraic Topology},  McGraw-Hill, New York, 1966. 
\bibitem[Sp2]{Sp2} E.H.~Spanier, {\em Singular homology  and cohomology with local coefficients and duality for manifolds}, Pacific J. of Math., {\bf 160}, (1993), 165-200.



\bibitem[Sw]{sw}  M.~Schwarz, {\em Equivalences for Morse homology}, in Geometry and Topology in Dynamics,
pp. 197-216, Contemporary Mathematics 246, M. Barge and K. Kuperberg, Eds,. AMS, 
(1999).


\end{thebibliography}
}
\end{document}